\documentclass[a4paper]{amsart}

\usepackage{a4wide}
\usepackage[UKenglish]{babel}
\usepackage{amsmath,amsfonts,amsthm}
\usepackage{graphicx}
	\graphicspath{{./}{figures/}}
\usepackage{enumerate}
\usepackage{color}

\theoremstyle{remark}\newtheorem{remark}{Remark}[section]
\theoremstyle{plain}
	\newtheorem{theorem}[remark]{Theorem}

\theoremstyle{definition}
	\newtheorem{definition}[remark]{Definition}
	
\allowdisplaybreaks

\newcommand{\abs}[1]{\left\vert#1\right\vert}
\newcommand{\cB}{\mathcal{B}}
\newcommand{\cE}{\mathcal{E}}
\newcommand{\cI}{\mathcal{I}}
\newcommand{\cM}{\mathcal{M}}
\newcommand{\cS}{\mathcal{S}}
\newcommand{\cT}{\mathcal{T}}
\newcommand{\cW}{\mathcal{W}}
\newcommand{\cV}{\mathcal{V}}
\newcommand{\dual}[2]{\langle#1,\,#2\rangle}
\newcommand{\norm}[1]{\left\Vert#1\right\Vert}
\newcommand{\normBL}[1]{\left\Vert#1\right\Vert^\ast_{BL}}
\newcommand{\R}{{\mathbb R}}
\newcommand{\supp}[1]{\operatorname{supp}#1}
\newcommand{\vmax}{v_\text{max}}

\begin{document}

\title{Transport of measures on networks}

\author{Fabio Camilli}
\address{Dipartimento di Scienze di Base e Applicate per l'Ingegneria, ``Sapienza'' Universit{\`a}  di Roma, Via Scarpa 16, 00161 Rome, Italy}
\email{camilli@sbai.uniroma1.it}

\author{Raul De Maio}
\address{Dipartimento di Scienze di Base e Applicate per l'Ingegneria, ``Sapienza'' Universit{\`a}  di Roma, Via Scarpa 16, 00161 Rome, Italy}
\email{raul.demaio@sbai.uniroma1.it}

\author{Andrea Tosin}
\address{Department of Mathematical Sciences ``G. L. Lagrange'', Politecnico di Torino, Corso Duca degli Abruzzi 24, 10129 Turin, Italy}
\email{andrea.tosin@polito.it}



\begin{abstract}
In this paper we formulate a theory of measure-valued linear transport equations on networks. The building block of our approach is the initial/boundary-value problem for the measure-valued linear transport equation on a bounded interval, which is the prototype of an arc of the network. For this problem we give an explicit representation formula of the solution, which also considers the total mass flowing out of the interval. Then we construct the global solution on the network by gluing all the measure-valued solutions on the arcs by means of appropriate distribution rules at the vertexes. The measure-valued approach makes our framework suitable to deal with multiscale flows on networks, with the microscopic and macroscopic phases represented by Lebesgue-singular and Lebesgue-absolutely continuous measures, respectively, in time and space.
\end{abstract}

\maketitle

\section{Introduction}
\label{sect:intro}
In recent times there has been an increasing interest in the notion of measure-valued solutions to evolution equations. Compared to standard approaches based on classical and weak solutions, the measure-theoretic setting allows one to better describe some interesting phenomena such as aggregation, congestion and pattern formation in a multiscale perspective. Several of these phenomena occur in applications such as vehicular traffic, data transmission, crowd motion, supply chains, where the state of the system evolves on a network, see e.g.~\cite{camilli2016SICON,dapice2009QAM,fermo2015M3AS,garavello2006BOOK,mugnolo2014BOOK}.

In order to extend the measure-valued approach to these irregular geometric structures, in this paper we study measure-valued solutions to a linear transport process defined on a network. For classical and weak solutions to transport equations on networks we refer the reader for example to~\cite{engel2008NHM,garavello2006BOOK,pokornyi2004JMS}.

The measure-valued approach in Euclidean spaces relies on the notion of push-forward of measures along the trajectories of a vector field describing the transport paths~\cite{ambrosio2008BOOK,canizo2011M3AS,cristiani2014BOOK,piccoli2014ARMA}. The study of these problems in bounded domains poses additional difficulties, especially concerning the behaviour at the boundaries of the transported measure. For problems on networks similar difficulties arise at the vertexes.

Our analysis is inspired by the results in~\cite{evers2015JDE,evers2016SIMA}, where measure-valued transport equations are studied in a bounded interval. We also refer to~\cite{gwiazda2012SIMA}, where the authors consider instead measure-valued solutions to non-linear transport problems with measure transmission conditions at nodal points, i.e. points where the velocity vanishes.

Consider a network $\Gamma=(\cV,\,\cE)$, where $\cV=\{V_i\}_{i\in I}$ is the set of vertexes and $\cE=\{E_j\}_{j\in J}$ is the set of arcs. We assume that the network is oriented and that a strictly positive, autonomous and Lipschitz continuous velocity field $v_j$ is defined on each arc $E_j$. Our aim is to describe the evolution of a mass distribution on the network $\Gamma$ transported by the velocity field $v(x)=\sum_{j\in J}v_j(x)\chi_{E_j}(x)$. For this we will make extensive use of the fundamental fact that a generic measure $\mu$ can be written as the superposition of elementary Dirac masses, i.e.
\begin{equation}
	\mu=\int_{\supp{\mu}}\delta_{x}\,d\mu(x),
	\label{superpos}
\end{equation}
where $\supp{\mu}$ denotes the support of $\mu$ belonging to an appropriate $\sigma$-algebra. This representation formula has to be understood in the sense of Bochner integrals.

From~\eqref{superpos} it follows that if we are able to define the transport of an atomic measure $\delta_x$ on the network then by linearity we can transport the whole distribution $\mu$. Hence, let us assume that the mass distribution $\mu_0$ at the initial time $t=0$ is given by a Dirac measure $\delta_{x_0}$, with $x_0\in E_j$ for some $j\in J$. If we postulate the conservation of the mass then in the time interval $(0,\,\tau)$ where the mass remains inside the arc $E_j$ the evolution of $\mu_0$ is governed by the continuity equation
\begin{equation}
	\partial_t\mu^j_t+\partial_x(v_j(x)\mu^j_t)=0,
	\label{CE}
\end{equation}
$\mu^j_t$ being a spatial measure denoting the mass distribution along the arc $E_j$ at time $t$.

For $t<\tau$ the solution to~\eqref{CE} is given by the push-forward of $\mu_0$ by means of the flow map
$$ \Phi^j_t(0,\,x_0):=x_0+\int_0^t v_j(\Phi^j_s(0,\,x_0))\,ds, $$
which describes the trajectory issuing from the point $x_0$ at time $t=0$ and arriving at the point $\Phi^j_t(x_0,\,0)\in E_j$ at time $t$. Consequently, $\mu^j_t$ is characterised as $\mu^j_t(A)=\mu_0((\Phi^j_t)^{-1}(A))$ for any measurable set $A\subseteq E_j$. Hence if  $\mu_0=\delta_{x_0}$ then $\mu^j_t=\delta_{\Phi^j_t(x_0,\,0)}$ for $t\in (0,\,\tau)$.

At $t=\tau$ the trajectory $t\mapsto\Phi^j_t(x_0,\,0)$ hits the final vertex $V_i$ of the arc $E_j$. Assuming that mass concentration at the vertexes of the network is not admitted, fractions $p^i_{jk}$ of the mass carried by $\delta_{\Phi^j_\tau(x_0,\,0)}$ have then to be distributed on each outgoing arc $E_k$ which originates from $V_i$.

This preliminary discussion sketches the main ideas that we intend to follow in order to tackle the global problem on the network. We first consider a local problem, namely a transport equation on each single arc with a measure acting as a source term (boundary condition) at the initial vertex. For this local problem we formulate an appropriate notion of measure-valued solution, for which we give a representation formula taking into account also the mass which flows out of the arc. Then we glue all the solutions on the single arcs by means of appropriate mass distribution rules at the vertexes, thereby constructing the global solution on the network.

In more detail, the paper is organised as follows. In Section~\ref{sec2} we introduce some notations and assumptions for the problem, while in Section~\ref{sec:measure} we review some basic facts about the measure-theoretic setting in which we will frame our analysis. In Section~\ref{sec:interval} we study the initial/boundary-value problem for the transport equation on a single bounded interval, which is the prototype of an arc of the network, then in Section~\ref{sec:network} we move to the problem on networks. Finally, in Section~\ref{sec:examples} we construct explicit measure-valued solutions on simple networks, which constitute preliminary examples of the application of our theory to vehicular traffic.

\section{Preliminary definitions and statement of the problem} \label{sec2}
We start by describing the constitutive elements of the problem.

\begin{definition}[Network]
A \emph{network} $\Gamma$ is a pair $(\cV,\,\cE)$ where $\cV:=\{V_i\}_{i\in I}$ is a finite collection of vertexes and $\cE:=\{E_j\}_{j\in J}$ is a finite collection of continuous non-self-intersecting oriented arcs whose endpoints belong to $\mathcal{V}$. Each arc $E_j$ is parameterised by a smooth function $\pi_j:[0,1]\to\R^n$. We assume that the network is connected and equipped with the topology induced by the minimum path distance.

Given a vertex $V_i\in\cV$, we say that an arc $E_j\in\cE$ is \emph{outgoing} (respectively, \emph{incoming}) if $V_i=\pi_j(0)$ (respectively, if $V_i=\pi_j(1)$). We denote by $d_O^i$ (respectively, by $d_I^i$) the number of outgoing (respectively, incoming) arcs in $V_i$ and by $d^i:=d_I^i+d_O^i$ the \emph{degree} of $V_i$. We say that a vertex $V_i$ is \emph{internal} if $d_I^i\cdot d_O^i>0$, that it is a \emph{source} if $d_O^i=d^i$ and finally that it is a \emph{well} if $d_I^i=d^i$.

We denote by $\cI$ the set of indexes $i\in I$ corresponding to the internal vertexes, by $\cS$ the one corresponding to the sources and by $\cW$ the one corresponding to the wells.
\end{definition}

\begin{definition}[Distribution matrices]
For an internal  vertex $V_i$, $i\in\cI$, and for $t>0$ we consider a \emph{distribution} (or \emph{transition}) \emph{matrix} $\{p^{i}_{kj}(t)\}_{k,\,j=1}^{d_I^i,\,d_O^i}$ such that
\begin{align}
	& p^{i}_{kj}(t)\geq 0 \nonumber \\
	& \sum_{j=1}^{d_{O}^i}p^i_{kj}(t)=\sum_{j\,:\,V_i=\pi_j(0)}p^i_{kj}(t)=1. \label{eq:pijk.sum.1}
\end{align}
Here $p^i_{kj}(t)$ represents the fraction of mass which at time $t$ flows from the incoming arc $E_k$ to the outgoing arc $E_j$ through the vertex $V_i$. Condition~\eqref{eq:pijk.sum.1} corresponds to the fact that, unlike~\cite{evers2015JDE,evers2016SIMA,gwiazda2012SIMA}, the mass cannot concentrate at the vertexes of the network.

For a source vertex $V_i$, $i\in\cS$, we consider instead a \emph{distribution vector} $\{p^{i}_{j}(t)\}_{j=1}^{d_O^i}$ such that
\begin{align}
	& p^{i}_{j}(t)\geq 0 \nonumber \\
	& \sum_{j=1}^{d_{O}^i}p^i_{j}(t)=\sum_{j\,:\,V_i=\pi_j(0)}p^i_j(t)=1. \label{eq:pij.sum.1}
\end{align}
\end{definition}

\begin{definition}[Velocity field]
On each arc $E_j\in\cE$ we assume that a strictly positive, bounded and Lipschitz continuous velocity $v_j:[0,\,1]\to (0,\,\vmax]$ is defined, with $0<\vmax<+\infty$. We denote by $v=\sum_{j\in J}v_j\chi_{E_j}$ the velocity field on the network ($\chi_{E_j}$ being the characteristic function of the arc $E_j$).
\end{definition}

\begin{definition}[Initial and boundary data]
We prescribe the initial mass distribution over $\Gamma$ as a positive measure $\mu_0=\sum_{j\in J}\mu_0^j$ with $\supp{\mu_0^j}\subseteq E_j$ for all $j$. Furthermore, at all the source vertexes $V_i$, $i\in\cS$, we prescribe an inflow measure $\varsigma^i$ with $\supp{\varsigma^i}\subseteq [0,\,T]$, $T>0$ being a certain final time.
\end{definition}

To define the transport of the initial measure $\mu_0$ and of the inflow measures $\{\varsigma^i\}_{i\in\cS}$ on the network $\Gamma$ we describe their evolution inside an arc. On each arc $E_j$ we take into account the inflow mass coming from the initial vertex $\pi_j(0)$ and we describe how the outflow mass leaving from the final vertex $\pi_j(1)$ is distributed to the corresponding outgoing arcs. In detail, we fix a final time $T>0$ and we consider the following system of measure-valued differential equations on $\Gamma\times [0,\,T]$:
\begin{equation}
	\begin{cases}
		\partial_{t}\mu^{j}+\partial_x(v_j(x)\mu^j)=0 & x\in E_j,\,t\in (0,\,T],\,j\in J \\[2mm]
		\mu_{t=0}^j=\mu_0^j & x\in E_j,\,j\in J \\[2mm]
		\mu^{j}_{V_i=\pi_j(0)}=
			\begin{cases}
				\sum\limits_{k=1}^{d_i^I}p^i_{kj}(t)\mu_{V_i=\pi_k(1)}^k & \text{if } i\in\cI \\[3mm]
				p^i_j(t)\varsigma^i & \text{if } i\in\cS,
			\end{cases} \\
	\end{cases}
	\label{problemontheroad}
\end{equation}
where by $\mu^j_{V_i=\pi_j(0)}$ we mean the measure flowing into the arc $E_j$ from its initial vertex $V_i=\pi_j(0)$ while by $\mu_{V_i=\pi_k(1)}^k$ we mean the measure flowing out of the arc $E_k$ from its final vertex $V_i=\pi_k(1)$. Moreover, by $p^{i}_{kj}(t)\mu_{V_i=\pi_k(1)}^k$ we mean a measure (in time) which is absolutely continuous with respect to $\mu_{V_i=\pi_k(1)}^k$ with density $p^{i}_{kj}(t)$ (and analogously for $p^i_j(t)\varsigma^i$).

For an internal vertex, the inflow measure is given by the mass flowing in $E_j$ from the arcs incident to $V_i=\pi_j(0)$ according to the distribution rule given by the distribution matrix $\{p^i_{kj}(t)\}$. For a source vertex, the inflow measure is the fraction $p^{i}_{j}(\cdot)$ of the prescribed datum $\varsigma^i$ entering $E_j$. The outflow measure, i.e. the part of the mass leaving the arc from the final vertex $\pi_j(1)$, is not given a priori but depends on the evolution of the measure $\mu$ inside the arc.

The detailed study of problem~\eqref{problemontheroad} is postponed to Section~\ref{sec:network}. Before that, we introduce an appropriate measure theoretic setting, see Section~\ref{sec:measure}, and consider preliminarily the problem on a single arc, see Section~\ref{sec:interval}.

\section{Measures and norms}
\label{sec:measure}
We introduce a space of measures with an appropriate norm where we consider the solutions to our measure-valued transport equations. Moreover, since the notion of solution is based on the superposition principle~\eqref{superpos}, we briefly describe the measure-theoretic setting which guarantees the validity of this formula. We refer for details to~\cite{ambrosio2008BOOK,bogachev2007BOOK,evers2015JDE,worm2010PhD}.

Let $\cT$ be a topological space with $\cB(\cT)$ the Borel $\sigma$-algebra in $\cT$. We denote by $\cM(\cT)$ the space of finite Borel measures on $\cT$ and by $\cM^{+}(\cT)$ the convex cone of the positive measures in $\cM(\cT)$. For $\mu\in\cM(\cT)$ and a bounded measurable function $\varphi:\cT\to\R$ we write
$$ \dual{\mu}{\varphi}:=\int_{\cT}\varphi\,d\mu. $$

Given a Borel measurable vector field $\Phi:\cT\to\cT$, the \emph{push-forward} of the measure $\mu$ under the action of $\Phi$ is an operation on $\mu$ which produces the new measure $\Phi\#\mu\in\cM(\cT)$ defined by
$$ (\Phi\#\mu)(E):=\mu(\Phi^{-1}(E)), \qquad \forall\,E\in\cB(\cT). $$
We immediately observe that $\dual{\Phi\#\mu}{\varphi}=\dual{\mu}{\varphi\circ\Phi}$.

Given a metric $d:\cT\times\cT\to\R_+$ in $\cT$, we denote by $BL(\cT)$ the Banach space of the bounded and Lipschitz continuous functions $\varphi:\cT\to\R$
equipped with the norm
$$ \norm{\varphi}_{BL}:=\norm{\phi}_{\infty}+\abs{\phi}_{L}, $$
where the semi-norm $\abs{\cdot}_L$ is defined by
$$ \abs{\varphi}_{L}:=\sup_{\substack{x,\,y\in\cT \\ x\ne y}}\frac{\abs{\varphi(y)-\varphi(x)}}{d(x,\,y)}. $$
Furthermore, we introduce a norm in $\cM(\cT)$ by taking the dual norm of $\norm{\cdot}_{BL}$:
$$ \normBL{\mu}:=\sup_{\substack{\varphi\in BL(\cT) \\ \norm{\varphi}_{BL}\leq 1}}\dual{\mu}{\varphi}. $$
It is easy to see that if $\mu\in\cM^+(\cT)$ then $\normBL{\mu}=\mu(\cT)$.

The space $(\cM(\cT),\,\normBL{\cdot})$ is in general \emph{not} complete, hence it is customary to consider its completion $\overline{\cM(\cT)}^{\normBL{\cdot}}$ with respect to the dual norm. However, the cone $\cM^+(\cT)$, which is a closed subset of $\cM(\cT)$ in the weak topology, is complete, although it is \emph{not} a Banach space because it is not a vector space. Since in our model we will consider only positive measures,   we restrict our attention to the complete metric space $(\cM^+(\cT),\,\normBL{\cdot})$ with the corresponding distance induced by the norm.

\begin{remark}
If $\cT$ is bounded the Kantorovich-Rubinstein's duality theorem implies that the norm $\normBL{\cdot}$ induces the \emph{$1$-Wasserstein distance} in $\cM^+(\cT)$.
\end{remark}

\begin{remark} The distance induced in $\cM(\cT)$ by the \emph{total variation norm}:
$$ \norm{\mu}_{TV}:=\sup_{\substack{\varphi\in C_b(\cT) \\ \norm{\varphi}_\infty\leq 1}}\dual{\mu}{\varphi}, $$
where $C_b(\cT)$ is the space of bounded continuous function on $\cT$, is another metric frequently used for measures. However, we observe that it may not be fully suited to transport problems where one wants to measure the distance between flowing mass distributions. Indeed, if we consider two points $x,\,y\in\cT$, $x\ne y$, and the corresponding Dirac mass distributions $\delta_x,\,\delta_y\in\cM^+(\cT)$ centred at them we see that
$$ \normBL{\delta_y-\delta_x}\leq d(x,\,y), \qquad \norm{\delta_y-\delta_x}_{TV}=2. $$
Hence the two measures are closer and closer in the norm $\normBL{\cdot}$ as the points $x,\,y$ approach, which is consistent with the intuitive idea of transport of mass distributions; while they are not in the total variation norm, no matter how close the points $x,\,y$ are.
\end{remark}

As alredy anticipated in Section~\ref{sect:intro}, for the subsequent development of the theory we will extensively use the following fact linked to the concept of \emph{Bochner integral}~\cite{bogachev2007BOOK,worm2010PhD}: any $\mu\in\cM^+(\cT)$ can be represented as a (continuous) sum of elementary masses in the form
$$ \mu=\int_{\cT}\delta_x\,d\mu(x) $$
as a Bochner integral in $(\overline{\cM(\cT)}^{\normBL{\cdot}},\,\normBL{\cdot})$.

\bigskip

We now specialise the previous definitions to the case $\cT=\Gamma\times [0,\,T]$, where $\Gamma\subset\R^n$ is a network. In particular, we will call $x$ the variable in each arc of $\Gamma$ and $t$ the variable in the interval $[0,\,T]$. We equip $\Gamma\times [0,\,T]$ with the distance
$$ d(x,\,y)+\abs{t-s}, \qquad (x,\,t),\,(y,\,s)\in\Gamma\times [0,\,T], $$
$d$ being the shortest path distance on $\Gamma$.

We consider the Borel $\sigma$-algebra $\cB(\Gamma\times [0,\,T])$ given by the union of the Borel $\sigma$-algebras $\cB([0,\,1]\times [0,\,T])$ corresponding to each arc $E_j$ of $\Gamma$. Thus $A\in\cB(\Gamma\times [0,\,T])$ if $(\pi_j^{-1},\,\operatorname{Id})(A\cap (E_j\times [0,\,T]))\in\cB([0,\,1]\times [0,\,T])$ for all $j\in J$, where $\operatorname{Id}$ denotes the identity mapping.

A measure $\mu$ belongs to $\cM(\Gamma\times [0,\,T])$ if each of its restrictions $\mu^j:=\mu\llcorner (E_j\times [0,\,T])$, $j\in J$, is a finite Borel measure on $E_j\times [0,\,T]$. We define the cone $\cM^+(\Gamma\times [0,\,T])$ analogously.

For $\mu\in\cM^+(\Gamma\times [0,\,T])$ and a bounded measurable function $\varphi:\Gamma\times [0,\,T]\to\R$ we write
\begin{equation}
	\dual{\mu}{\varphi}:=\sum_{j\in J}\int_{E_j\times [0,\,T]}\varphi\,d\mu^j.
 	\label{pairing}
\end{equation}

For a function $\varphi:\Gamma\times [0,\,T]\to\R$, we denote by $\varphi_j:[0,\,1]\times [0,\,T]\to\R$ its restriction to $E_j\times [0,\,T]$, i.e.:
$$ \varphi(x,\,t)=\varphi_j(y,\,t) \quad \text{for } x\in E_j,\ y=\pi_j^{-1}(x),\ t\in [0,\,T]. $$

A function $\varphi$ belongs to $BL(\Gamma\times [0,\,T])$ if it is continuous on $\Gamma$ and $\varphi_j\in BL([0,\,1]\times [0,\,T])$ for every $j\in J$.
For $\varphi\in BL(\Gamma\times [0,\,T])$ the norm $\norm{\varphi}_{BL(\Gamma\times [0,\,T])}$ is defined by
$$ \norm{\varphi}_{BL(\Gamma\times [0,\,T])}:=\sup_{j\in J}\norm{\varphi_j}_{BL([0,\,1]\times [0,\,T])}. $$
The corresponding dual norm $\normBL{\cdot}$ of a measure $\mu\in\cM(\Gamma\times [0,\,T])$ is given by
$$ \normBL{\mu}:=\sup_{\substack{\varphi\in BL(\Gamma\times [0,\,T]) \\ \norm{\varphi}_{BL(\Gamma\times [0,\,T])}\leq 1}}\dual{\mu}{\varphi}. $$

\section{The transport equation in a bounded interval}
\label{sec:interval}
In this section  we study the transport equation in a bounded interval. Actually, we start by focusing on the problem of prescribing appropriate initial and boundary conditions to the differential equation in $\R^+\times\R^+$, which is an unbounded domain with boundary; then we will restrict the results to a truly bounded domain.

Consider the conservation law
\begin{equation}
   \partial_{t}\mu+\partial_{x}(v(x)\mu)=0, \qquad (x,\,t)\in\R^+\times\R^+,
   \label{CL_R}
\end{equation}
where $v:\R^+\to\R$ is a strictly positive, bounded and Lipschitz continuous velocity field, so that the flow is one-directional and depends only on the space variable $x$. Given $\mu\in\cM^+(\R^+_0\times\R^+_0)$, where $\R^+_0:=[0,\,+\infty)$, owing to the disintegration theorem~\cite[Section 5.3]{ambrosio2008BOOK} we can decompose this measure by means of its projection maps on the coordinate axes:
\begin{itemize}
\item using the projection with respect to the space variable we can write
\begin{equation}
	\mu(dx\,dt)=\mu_t(dx)\otimes dt,
	\label{eq:dec.space-time}
\end{equation}
where $dt$ is the Lebesgue measure in time in $\R^+_0$ and $\mu_t\in\cM^+(\R^+_0\times\{t\})\equiv\cM^+(\R^+_0)$ for a.e. $t\in\R^+_0$. The measure $\mu_t$ is called the \emph{conditional measure}, or \emph{trace}, of $\mu$ with respect to $t$ on the fibre $\R^+_0\times\{t\}$;

\item similarly, projecting with respect to the time variable we can write
\begin{equation}
	\mu(dx\,dt)=\frac{\nu_x(dt)}{v(x)}\otimes dx,
	\label{eq:dec.time-space}
\end{equation}
where $dx$ is the Lebesgue measure in space in $\R^+_0$ and $\nu_x\in\cM^+(\{x\}\times\R^+_0)\equiv\cM^+(\R^+_0)$ for a.e. $x\in\R^+_0$. The measure $\nu_x$ is called the \emph{conditional measure}, or \emph{trace}, of $\mu$ with respect to $x$ on the fibre $\{x\}\times\R^+_0$.
\end{itemize}

\begin{remark}
The coefficient $\frac{1}{v(x)}$ in the decomposition~\eqref{eq:dec.time-space} is considered for dimensional reasons, so that $\nu_x$ represents actually the mass distributed on the fibre $\{x\}\times\R^+_0$.

We incidentally notice that if $\mu$ solves~\eqref{CL_R} then the mapping $x\mapsto\nu_x$ solves the equation $\partial_x\nu_x+\bar{\partial}_t\nu_x=0$, where $\bar{\partial}_t:=\frac{1}{v(x)}\partial_t$. As far as the decomposition~\eqref{eq:dec.space-time} is concerned, the mapping $t\mapsto\mu_t$ solves instead the equation $\partial_t\mu_t+\partial_x(v(x)\mu_t)=0$.
\end{remark}

Relying on the concept of conditional measures, we formulate the following initial/boundary-value problem for~\eqref{CL_R}:
\begin{equation}
	\begin{cases}
		\partial_{t}\mu+\partial_{x}(v(x)\mu)=0 & (x,\,t)\in\R^+\times\R^+ \\
		\mu_{t=0}=\mu_0\in\cM^+(\R^+_0\times\{0\}) \\
		\nu_{x=0}=\nu_0\in\cM^+(\{0\}\times\R^+_0)
	\end{cases}
	\label{problemsingle}
\end{equation}
with $\mu\in\cM^+(\R^+_0\times\R^+_0)$, where:
\begin{itemize}
\item assigning an initial condition at $t=0$ amounts to prescribing the trace of $\mu$ on the fibre $\R^+_0\times\{0\}$ according to the decomposition~\eqref{eq:dec.space-time};
\item assigning a boundary condition at $x=0$ amounts to prescribing the trace of $\mu$ on the fibre $\{0\}\times\R^+_0$ according to the decomposition~\eqref{eq:dec.time-space}.
\end{itemize}
In order to give a suitable notion of measure-valued solution to~\eqref{problemsingle}, we preliminarily introduce integration-by-parts formulas useful to deal with the initial and boundary data. Let $C^1_0(\R^+_0\times\R^+_0)$ be the space of continuous functions in $\R^+_0\times\R^+_0$ which are differentiable in $\R^+\times\R^+$ and vanish for $x,\,t\to +\infty$. For $\mu\in\cM^+(\R^+_0\times\R^+_0)$ and $\varphi\in C^1_0(\R^+_0\times\R^+_0)$ we set:
\begin{align*}
	\dual{\partial_t\mu}{\varphi} &:= -\dual{\mu}{\partial_t\varphi}-\int_{\R^+_0}\varphi(x,\,0)\,d\mu_0(x), \\
	\dual{\partial_x(v(x)\mu)}{\varphi} &:= -\dual{\mu}{v(x)\partial_x\varphi}-\int_{\R^+_0}\varphi(0,\,t)\,d\nu_0(t),
\end{align*}
where $\dual{\cdot}{\cdot}$ denotes the duality pairing between measures and test functions in $\R^+_0\times\R^+_0$, i.e. $\dual{\mu}{\varphi}=\iint_{\R^+_0\times\R^+_0}\varphi(x,\,t)\,d\mu(x,\,t)$. Notice that if $\varphi$ is compactly supported in $\R^+\times\R^+$ then the previous formulas agree with the usual definition of the distributional derivatives of $\mu$.

\begin{remark}
With a slight abuse of notation, in the following we will denote
$$ \int_{\R^+_0}\varphi(x,\,0)\,d\mu_0(x)=:\dual{\mu_0}{\varphi}, \qquad
	\int_{\R^+_0}\varphi(0,\,t)\,d\nu_0(t)=:\dual{\nu_0}{\varphi}, $$
the difference between duality pairings in $\R^+_0\times\R^+_0$ and in $\R^+_0\times\{0\}$ or $\{0\}\times\R^+_0$ being clear from the measures used.
\end{remark}

Thanks to these formulas, we are in a position to introduce the following notion of measure-valued solution to~\eqref{problemsingle}:
\begin{definition} \label{def_weak}
Given $\mu_0\in\cM^+(\R^+_0\times\{0\})$ and $\nu_0\in\cM^+(\{0\}\times\R^+_0)$, a measure-valued solution to~\eqref{problemsingle} is a finite measure $\mu\in\cM^+(\R^+_0\times \R^+_0)$ such that
\begin{equation}
	\dual{\mu}{\partial_t\varphi+v(x)\partial_x\varphi}=-\dual{\mu_0}{\varphi}-\dual{\nu_0}{\varphi}, \quad
		\forall\,\varphi\in C^1_0(\R^+_0\times\R^+_0).
	\label{weakcond}
\end{equation}
\end{definition}

Since~\eqref{problemsingle} is a linear problem, its solution can be obtained from the superposition of two measures $\mu^1,\,\mu^2\in\cM^+(\R^+_0\times \R^+_0 )$, where $\mu^1$ is the solution to~\eqref{problemsingle} with data $\mu_{t=0}=\mu_0$ and $\nu_{x=0}=0$ while $\mu^2$ is the solution to~\eqref{problemsingle} with data $\mu_{t=0}=0$ and $\nu_{x=0}=\nu_0$. This is doable in a standard way in terms of the push-forward of the initial and boundary data by means of appropriate vector fields in $\R^+_0\times\R^+_0$, cf.~\cite{ambrosio2008BOOK}. With this approach time and space play the same role, the former being understood in particular as an additional state variable of the system.

However, for the next purposes it is convenient to characterise the solution $\mu$ to~\eqref{problemsingle} by means of the traces of $\mu^1$ and $\mu^2$ over the fibres $\R^+_0\times\{t\}$, $t>0$; i.e.
$$ \mu(dx\,dt)=(\mu^1_t(dx)+\mu^2_t(dx))\otimes dt, $$
where $\mu^1_t$, $\mu^2_t$ are given by the transport of $\mu_0$, $\nu_0$, respectively, along the characteristics generated in $\R^+\times\R^+$ by the velocity field $v$.

In order to obtain a formula for $\mu^1_t$, let $\Phi_t=\Phi_t(x,\,0)$ be the position at time $t>0$ of the particle which is in $x\in\R^+_0$ at time $t=0$ and which moves following the velocity field $v=v(x)$:
\begin{equation}
	\begin{cases}
		\dfrac{d}{dt}\Phi_t(x,\,0)=v(\Phi_t(x,\,0)), & t>0 \\[2mm]
		\Phi_0(x,\,0)=x.
	\end{cases}
	\label{eq:char1}
\end{equation}
By standard results, it is well known that
$$ \mu_t^1=\Phi_t\#\mu_0=\int_{\R^+_0}\delta_{\Phi_t(x,\,0)}\,d\mu_0(x)\in\cM^+(\R^+_0\times\{t\}), $$
where $\#$ is the push-forward operator, $\delta$ is the Dirac delta measure, and the integral at the right-hand side is understood in the sense of Bochner.

Likewise, to obtain a formula for $\mu^2_t$ we consider the characteristic lines issuing from the $t$ axis. In particular, we denote now by $\Phi_t(0,\,s)$ the position at time $t>0$ of the particle which is in $x=0$ at time $s\in\R^+_0$ and which moves following the velocity field $v=v(x)$:
 \begin{equation}
	\begin{cases}
		\dfrac{d}{dt}\Phi_t(0,\,s)=v(\Phi_t(0,\,s)), & t>s \\[2mm]
		\Phi_s(0,\,s)=0.
	\end{cases}
	\label{eq:char2}
\end{equation}
By transporting the mass $\nu_0$ along these characteristics we can write
$$ \mu^2_t=\int_{[0,\,t]}\delta_{\Phi_t(0,\,s)}\,d\nu_0(s)\in\cM^+(\R^+_0\times \{t\}), $$
where the integral is again meant in the sense of Bochner.

Summing up, we consider the following representation formula for $\mu$:
\begin{equation}
	\mu(dx\,dt)=\left(\int_{\R^+_0}\delta_{\Phi_t(\xi,\,0)}(dx)\,d\mu_0(\xi)
		+\int_{[0,\,t]}\delta_{\Phi_t(0,\,s)}(dx)\,d\nu_0(s)\right)\otimes dt
	\label{solcand1}
\end{equation}
and we check that it actually defines a solution to~\eqref{problemsingle} in the sense of Definition~\ref{def_weak}. To this purpose we preliminarily observe that, since $\mu_t^1=\Phi_t\#\mu_0$, for every (bounded and measurable) function $f:\R^+_0\to\R$ it results
\begin{equation}
	\int_{\R^+_0}f(x)\,d\mu^1_t(x)=\int_{\R_0^+}f(\Phi_t(x,\,0))\,d\mu_0(x).
	\label{c1_a}
\end{equation}
We can obtain a similar formula for $\mu^2_t$ by observing that, given a simple function $f:\R^+_0\to\R$, $f(x)=\sum_{k=1}^{N}\alpha_k\chi_{A_k}(x)$, where $\{A_k\}_{k=1}^N$ is a measurable finite disjoint partition of $\R^+_0$, it results
\begin{align*}
	\int_{\R^+_0}f(x)\,d\mu^2_t(x) &= \sum_{k=1}^{N}\alpha_k \mu^2_t(A_k)
		= \sum_{k=1}^{N}\alpha_k\int_{[0,\,t]}\delta_{\Phi_t(0,\,s)}(A_k)\,d\nu_0(s) \\
	&= \sum_{k=1}^{N}\alpha_k\int_{[0,\,t]}\chi_{A_k}(\Phi_t(0,\,s))\,d\nu_0(s)
		=\int_{[0,\,t]}\sum_{k=1}^{N}\alpha_k\chi_{A_k}(\Phi_t(0,\,s))\,d\nu_0(s) \\
	&= \int_{[0,\,t]}f(\Phi_t(0,\,s))\,d\nu_0(s).
\end{align*}
Approximating a measurable function $f$ with a sequence of simple functions we get in general
\begin{equation}
	\int_{\R^+_0}f(x)\,d\mu^2_t(x)=\int_{[0,\,t]}f(\Phi_t(0,\,s))\,d\nu_0(s).
	\label{c1_b}
 \end{equation}
Interestingly, an integral with respect to the $x$ variable is converted into one with respect to the $t$ variable.

Plugging~\eqref{solcand1} into the left-hand side of~\eqref{weakcond} and using~\eqref{c1_a},~\eqref{c1_b} we discover:
\begin{align*}
	\dual{\mu}{\partial_t\varphi+v(x)\partial_x\varphi} &=
		\int_{\R^+_0}\int_{\R^+_0}\Bigl(\partial_t\varphi(\Phi_t(x,\,0),\,t)+v(\Phi_t(x,\,0))\partial_x\varphi(\Phi_t(x,\,0),\,t)\Bigr)\,d\mu_0(x)\,dt \\
		&\phantom{=} +\int_{\R^+_0}\int_{[0,\,t]}(\partial_t\varphi(\Phi_t(0,\,s),\,t)+v(\Phi_t(0,\,s))\partial_x\varphi(\Phi_t(0,\,s),\,t))\,d\nu_0(s)\,dt \\
	&= \int_{\R^+_0}\int_{\R^+_0}\frac{d}{dt}\varphi(\Phi_t(x,\,0),\,t)\,d\mu_0(x)\,dt
		+\int_{\R^+_0}\int_{[0,\,t]}\frac{d}{dt}\varphi(\Phi_t(0,\,s),\,t)\,d\nu_0(s)\,dt,
\intertext{where in the last passage we have invoked~\eqref{eq:char1},~\eqref{eq:char2}. By switching the order of integration in view of Fubini-Tonelli's Theorem we further obtain}
	&= \int_{\R^+_0}\int_{\R^+_0}\frac{d}{dt}\varphi(\Phi_t(x,\,0),\,t)\,dt\,d\mu_0(x)
		+ \int_{\R^+_0}\int_{[s,\,+\infty)}\frac{d}{dt}\varphi(\Phi_t(0,\,s),\,t)\,dt\,d\nu_0(s) \\
	&=\int_{\R^+_0}\Bigl[\varphi(\Phi_t(x,\,0),\,t)\Bigr]_{t=0}^{t=+\infty}\,d\mu_0(x)
		+\int_{\R^+_0}\Bigl[\varphi(\Phi_t(0,\,s),\,t)\Bigr]_{t=s}^{t=+\infty}\,d\nu_0(s) \\
	&= -\int_{\R^+_0}\varphi(x,\,0)\,d\mu_0(x)-\int_{\R^+_0}\varphi(0,\,s)\,d\nu_0(s) \\
	&= -\dual{\mu_0}{\varphi}-\dual{\nu_0}{\varphi},
\end{align*}
which confirms that~\eqref{solcand1} is indeed a measure-valued solution to~\eqref{problemsingle}. Uniqueness of such a solution is a consequence of continuous dependence estimates on the initial and boundary data, which can be proved by standard arguments in literature, cf.~\cite{ambrosio2008BOOK}. In conclusion, for the transport problem in $\R^+\times\R^+$ we have the following well-posedness result:
\begin{theorem} \label{T1}
For $\mu_0\in\cM^+(\R^+_0\times\{0\})$, $\nu_0\in\cM^+(\{0\}\times\R^+_0)$ there exists a unique measure-valued solution to~\eqref{problemsingle} in the sense of Definition~\ref{def_weak}, which can be represented by~\eqref{solcand1}.
\end{theorem}

We now pass to consider the transport problem on the bounded domain $Q:=(0,\,1)\times (0,\,T)$, $T>0$, i.e.
\begin{equation}
	\begin{cases}
		\partial_t\mu+\partial_x(v(x)\mu)=0, & (x,\,t)\in Q \\
		\mu_{t=0}=\mu_0\in\cM^+([0,\,1]\times\{0\}) \\
		\nu_{x=0}=\nu_0\in\cM^+(\{0\}\times [0,\,T])
	\end{cases}
	\label{problemsingleQ}
\end{equation}
for a given bounded, strictly positive and Lipschitz continuous velocity field $v:[0,\,1]\to (0,\,\vmax]$. The solution to this problem can be obtained by restricting to $Q$ the measure $\mu$ solving~\eqref{problemsingle} (with the velocity field $v$ possibly extended to the whole $\R^+_0$ as, e.g. $v(x)=v(1)$ for $x\ge 1$). Therefore we are going to consider the restriction of $\mu$ to $Q$ defined as the measure $\mu\llcorner Q$ such that
$$ \mu\llcorner Q(E):=\mu(E\cap Q) $$
for every measurable set $E\subseteq\R^+_0\times\R^+_0$.

In particular, in view of the application of this problem to a network, it is important to characterise the traces of $\mu\llcorner Q$ on the fibres $[0,\,1]\times\{T\}$ and $\{1\}\times[0,\,T]$, which depend on the transport of $\mu_0$ and $\nu_0$ inside $Q$.

Let us introduce the following quantities:
\begin{align}
	\tau(x) &:= \inf\{t\geq 0\,:\,\Phi_t(x,\,0)=1 \}, \quad x\in [0,\,1] \label{exitime} \\
	\sigma(s) &:= \inf\{t\geq s\,:\,\Phi_t(0,\,s)=1\}, \quad s\in [0,\,T] \label{exitime2}
\end{align}
corresponding to the time needed to the characteristic line issuing from either $(x,\,0)$, in case of $\tau(x)$, or $(0,\,s)$, in case of $\sigma(s)$, to hit the boundary $x=1$. Since the considered transport problem is linear, in particular the velocity field $v$ does not depend on the measure $\mu$ itself, both $\tau$ and $\sigma$ are one-to-one, thus invertible. Moreover $\tau$ decreases with $x$ while $\sigma$ increases with $s$ and, in particular, $\sigma(s)=\tau(0)+s$ because $v$ is autonomous.

Recalling~\eqref{solcand1} and using $\tau$, $\sigma$ we write the trace of $\mu\llcorner Q$ on the fibre $[0,\,1]\times\{T\}$ as (cf. Figure~\ref{fig:characteristics})
\begin{equation}
	\mu_{T}:=\int_{[0,\,\max\{0,\,\tau^{-1}(T)\}]}\,\delta_{\Phi_T(x,\,0)}\,d\mu_0(x)
		+\int_{[\max\{0,\,\sigma^{-1}(T)\},\,T]}\delta_{\Phi_T(0,\,s)}\,d\nu_0(s)
	\label{candidate}
\end{equation}
whereas, following the characteristics, we construct the trace on the fibre $\{1\}\times [0,\,T]$ as
\begin{equation}
	\nu_1:=\int_{(\max\{0,\,\tau^{-1}(T)\},\,1]}\delta_{\tau(x)}\,d\mu_0(x)
		+\int_{[0,\,\max\{0,\,\sigma^{-1}(T)\})}\delta_{\sigma(s)}\,d\nu_0(s).
	\label{outnu}
\end{equation}
We incidentally notice that the first term at the right-hand side of~\eqref{candidate} is the push-forward of $\mu_0$ by the flow map $\Phi_T$ then restricted to $x\in [0,\,1]$ while the second term at the right-hand side of~\eqref{outnu} is the push-forward of $\nu_0$ by the mapping $\sigma$ then restricted to $t\in [0,\,T]$.

\begin{figure}[!t]
\centering
\includegraphics[width=\textwidth]{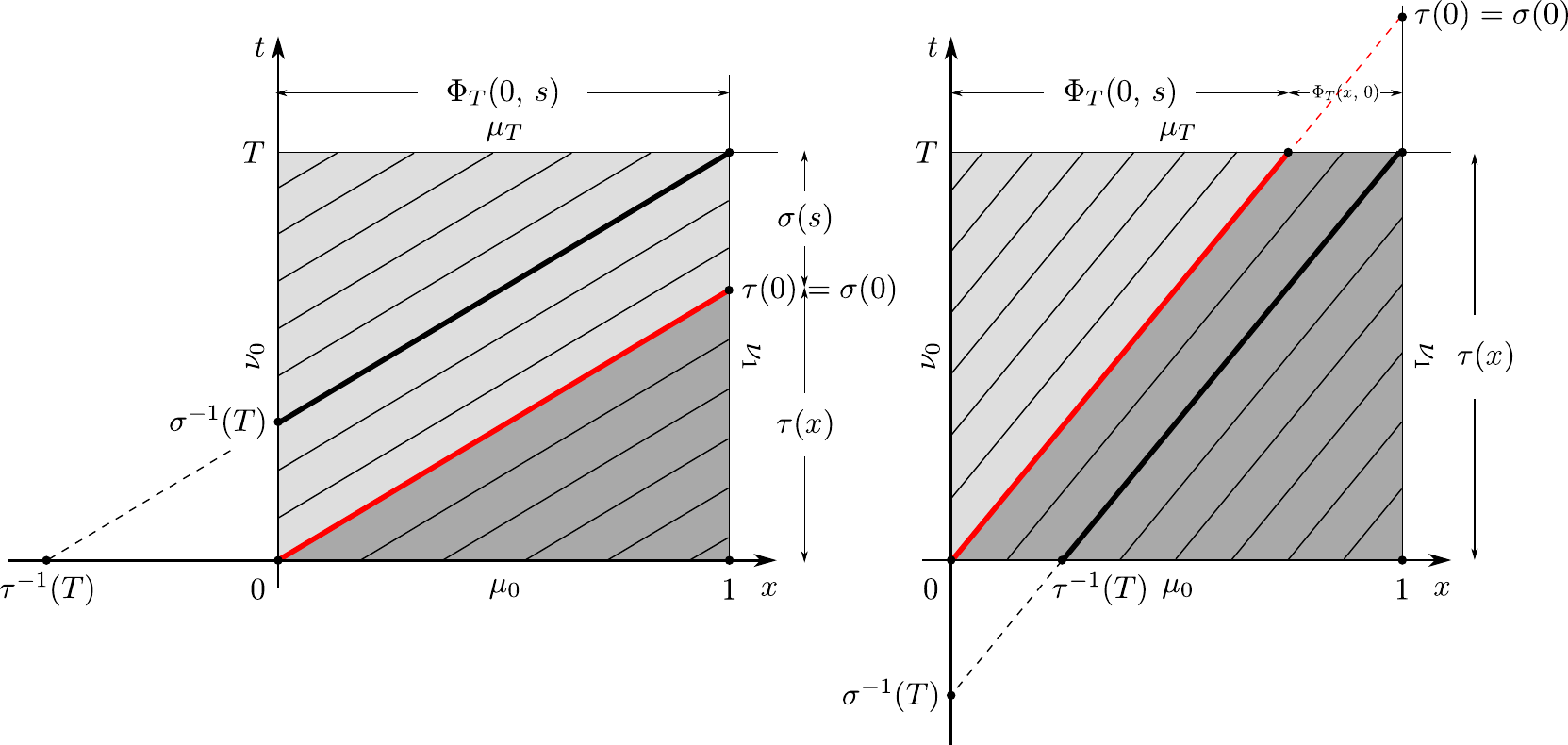}
\caption{Sketch of the characteristics of problem~\eqref{problemsingleQ} in the two cases $\tau(0)=\sigma(0)<T$ (left) and $\tau(0)=\sigma(0)>T$ (right). For pictorial purposes we imagine a constant velocity field, so that the characteristics are straight lines in the space-time.}
\label{fig:characteristics}
\end{figure}

The relationship between these traces and the transport of $\mu_0$, $\nu_0$ inside $Q$ is rigorously stated by the following theorem, which represents our main result on problem~\eqref{problemsingleQ}:
\begin{theorem} \label{existence}
Given $\mu_0\in\cM^+([0,\,1]\times\{0\})$, $\nu_0\in\cM^+(\{0\}\times[0,\,T])$, the measure $\mu\llcorner Q\in\cM^+(\bar{Q})$, $\mu\in\cM^+(\R^+_0\times\R^+_0)$ being the solution to~\eqref{problemsingle}, is the unique measure which satisfies the balance
\begin{equation}
	\dual{\mu\llcorner Q}{\partial_t\varphi+v(x)\partial_x\varphi}=\dual{\mu_T-\mu_0}{\varphi}+\dual{\nu_1-\nu_0}{\varphi},
		\quad \forall\,\varphi\in C^1(\bar{Q}),
	\label{psexit}
\end{equation}
where $\mu_T\in\cM^+([0,\,1]\times\{T\})$, $\nu_1\in\cM^+(\{1\}\times [0,\,T])$ are the traces defined in~\eqref{candidate},~\eqref{outnu}, respectively.

Moreover, for $\mu^{k}_0\in\cM^+([0,\,1]\times\{0\})$, $\nu_0^{k}\in\cM^+(\{0\}\times [0,\,T])$, $k=1,\,2$, there exists a constant $C=C(T)>0$ such that
\begin{equation}
	\normBL{\mu^2_T-\mu^1_T}+\normBL{\nu^2_1-\nu^1_1}\leq C\left(\normBL{\mu^2_0-\mu^1_0}+\normBL{\nu^2_0-\nu^1_0}\right).
	\label{dependence}
\end{equation}
\end{theorem}
\begin{proof}
See Appendix~\ref{appendix}.
\end{proof}

We also  give a result about the dependence on time.

\begin{theorem} \label{theo:time_dep}
Given $\mu_0\in\cM^+([0,\,1]\times\{0\})$, $\nu_0\in\cM^+(\{0\}\times [0,\,T])$, there exists a constant $C=C(T)>0$ such that
\begin{equation}
	\normBL{\mu_t-\mu_{t'}}+\normBL{\nu_1\llcorner [0,\,t]-\nu_1\llcorner [0,\,t']}\leq C\abs{t-t'}+\nu_0([t',\,t])
	\label{notcont}
\end{equation}
for all $t',\,t\in [0,\,T]$ with $t'<t$.
\end{theorem}
\begin{proof}
See Appendix~\ref{appendix}.
\end{proof}

\begin{remark}
Theorem~\ref{theo:time_dep} states virtually that the traces $\mu_t$ and $\nu_1\llcorner [0,\,t]$ of $\mu\llcorner Q$ are Lipschitz continuous in time, a part from the presence of the term $\nu_0([t',\,t])$ in the estimate~\eqref{notcont}.

If the boundary datum $\nu_0$ is absolutely continuous with respect to the Lebesgue measure in the interval $[t',\,t]$ then for $t\to t'$ we get actually $\normBL{\mu_t-\mu_{t'}}+\normBL{\nu_1\llcorner [0,\,t]-\nu_1\llcorner [0,\,t']}\to 0$. If instead $\nu_0$ contains singularities in $[t',\,t]$ then the distances $\normBL{\mu_t-\mu_{t'}}$, $\normBL{\nu_1\llcorner [0,\,t]-\nu_1\llcorner [0,\,t']}$ between two traces on horizontal and vertical fibres are in general not proportional to the time gap $\abs{t-t'}$.

In the applications, a Lebesgue-absolutely continuous $\nu_0$ corresponds to a \emph{macroscopic} inflow mass provided with density. A Lebesgue-singular $\nu_0$ corresponds instead to \emph{microscopic} point masses flowing from the boundary $x=0$ during the time interval $[t',\,t]$ and then propagating as singularities across $Q$.
\end{remark}

\section{The transport equation on a network}
\label{sec:network}
In this section we go back to the study of problem~\eqref{problemontheroad}. In order to make the notation consistent with the one introduced in Section~\ref{sec:interval}, we set
$$ \nu_0^j:=\mu^{j}_{V_i=\pi_j(0)}, \qquad \nu_1^j:=\mu^j_{V_i=\pi_j(1)} $$
and we rewrite  \eqref{problemontheroad} as
\begin{equation}
	\begin{cases}
		\partial_t\mu^{j}+\partial_x(v_j(x)\mu^j)=0 & x\in E_j,\,t\in (0,\,T],\,j\in J \\[2mm]
		\mu^j_{t=0}=\mu_0^j & x\in E_j,\,j\in J \\
		\nu_0^j=
			\begin{cases}
				\sum\limits_{k\,:\,V_i=\pi_k(1)}p^i_{kj}(t)\nu_1^k & \text{if } i\in\cI \\[3mm]
				p^i_j(t)\varsigma^i & \text{if } i\in\cS.
			\end{cases}
	\end{cases}
	\label{problemonnet}
\end{equation}

Let $\varphi\in C^1(\Gamma\times [0,\,T])$. Given $\mu_0^j\in\cM^+(E_j\times\{0\})$, $\nu_0^j\in\cM^+(\{0\}\times [0,\,T])$, owing to Theorem~\ref{existence} there exists $\mu^j\in\cM^+(E_j\times [0,\,T])$ such that
\begin{equation}
	\dual{\mu^j}{\partial_t\varphi+v_j(x)\partial_x\varphi}=\dual{\mu^j_T-\mu^j_0}{\varphi}+\dual{\nu^j_1-\nu^j_0}{\varphi}
	\label{eq:weak.on.arc}
\end{equation}
for every $j\in J$. Similarly to~\eqref{candidate},~\eqref{outnu}, the traces $\mu^j_T$, $\nu^j_1$ are
\begin{align}
    \mu^j_T &= \int_{[0,\,\max\{0,\,\tau_j^{-1}(T)\}]}\delta_{\Phi^j_T(x,\,0)}\,d\mu_0^j(x)
    		+\int_{[\max\{0,\,\sigma_j^{-1}(T)\},\,T]}\delta_{\Phi^j_T(0,\,s)}\,d\nu_0^j(s) \label{mu_j} \\
	\nu^j_1 &= \int_{(\max\{0,\,\tau_j^{-1}(T)\},\,1]}\delta_{\tau_j(x)}\,d\mu^j_0(x)
		+\int_{[0,\,\max\{0,\,\sigma_j^{-1}(T)\})}\delta_{\sigma_j(s)}\,d\nu^j_0(s), \label{nu_j}
\end{align}
where the flow maps $\Phi^j_t(x,\,0)$ and $\Phi^j_t(0,\,s)$ are defined like in~\eqref{eq:char1},~\eqref{eq:char2}, respectively, using the velocity field $v_j(x)$ on the arc $E_j$, $j\in J$. Likewise, $\tau_j$ and $\sigma_j$ are defined like in~\eqref{exitime},~\eqref{exitime2}.

Summing~\eqref{eq:weak.on.arc} over $j$ and recalling~\eqref{pairing} we deduce
\begin{equation}
	\dual{\mu}{\partial_t\varphi+v(x)\partial_x\varphi}=
		\dual{\mu_T-\mu_0}{\varphi}+\sum_{j\in J}\dual{\nu_1^j-\nu_0^j}{\varphi},
	\label{trdstep}
\end{equation}
where
\begin{equation}
	\mu_0=\sum_{j\in J}\mu_0^j, \qquad
		\mu_T=\sum_{j\in J}\mu^{j}_T.
	\label{eq:mu0T.network}
\end{equation}
In particular, the last term at the right-hand side in~\eqref{trdstep} can be rewritten in more detail by summing on the vertexes of the network:
\begin{align*}
	\sum_{j\in J}\dual{\nu_1^j-\nu_0^j}{\varphi} &=
		\sum_{i\in I}\left(\sum_{j\,:\,V_i=\pi_j(1)}\dual{\nu^j_1}{\varphi}
			-\sum_{j\,:\,V_i=\pi_j(0)}\dual{\nu^j_0}{\varphi}\right) \\
	&= \sum_{i\in\cI}\left(\sum_{j\,:\,V_i=\pi_j(1)}\dual{\nu^j_1}{\varphi}
		-\sum_{j\,:\,V_i=\pi_j(0)}\dual{\nu^j_0}{\varphi}\right) \\
	&\phantom{=} +\sum_{i\in\cW}\sum_{j\,:\,V_i=\pi_j(1)}\dual{\nu^j_1}{\varphi}
		-\sum_{i\in\cS}\sum_{j\,:\,V_i=\pi_j(0)}\dual{\nu^j_0}{\varphi}.
\end{align*}

For an internal vertex $V_i$, $i\in\cI$, using the corresponding boundary condition prescribed in~\eqref{problemonnet} we obtain:
\begin{align*}
	\sum_{j\,:\,V_i=\pi_j(1)}\dual{\nu^j_1}{\varphi}
		-\sum_{j\,:\,V_i=\pi_j(0)}\dual{\nu^j_0}{\varphi} &=
			\sum_{j\,:\,V_i=\pi_j(1)}\dual{\nu^j_1}{\varphi}
				-\sum_{j\,:\,V_i=\pi_j(0)}\dual{\sum_{k\,:\,V_i=\pi_k(1)}p^i_{kj}(t)\nu^k_1}{\varphi} \\
	&= \sum_{j\,:\,V_i=\pi_j(1)}\dual{\nu^j_1}{\varphi}
		-\sum_{k\,:\,V_i=\pi_k(1)}\dual{\sum_{j\,:\,V_i=\pi_j(0)}p^i_{kj}(t)\nu^k_1}{\varphi}
\intertext{whence, taking~\eqref{eq:pijk.sum.1} into account in the second term at the right-hand side,}
	&= \sum_{j\,:\,V_i=\pi_j(1)}\dual{\nu^j_1}{\varphi}
		-\sum_{k\,:\,V_i=\pi_k(1)}\dual{\nu^k_1}{\varphi} \\
	&= 0.
\end{align*}
This is the conservation of the mass through the internal vertexes of the network.

For a source vertex $V_i$, $i\in\cS$, we use the corresponding boundary condition prescribed in~\eqref{problemonnet} to find:
\begin{align*}
	\sum_{i\in\cS}\sum_{j\,:\,V_i=\pi_j(0)}\dual{\nu^j_0}{\varphi} &=
		\sum_{i\in\cS}\sum_{j\,:\,V_i=\pi_j(0)}\dual{p^i_j(t)\varsigma^i}{\varphi} \\
	&= \sum_{i\in\cS}\dual{\left(\sum_{j\,:\,V_i=\pi_j(0)}p^i_j(t)\right)\varsigma^i}{\varphi}
\intertext{whence, in view of~\eqref{eq:pij.sum.1},}
	&= \sum_{i\in\cS}\dual{\varsigma^i}{\varphi}=\dual{\varsigma}{\varphi}
\end{align*}
where we have defined the measure $\varsigma:=\sum_{i\in\cS}\varsigma^i\in\cM^+(\cup_{i\in\cS}\{V_i\}\times [0,\,T])$. This is the total mass flowing into the network from the source vertexes up to the time $T$.

Finally, for a well vertex $V_i$, $i\in\cW$, we define
\begin{equation}
	\omega^i:=\sum_{j\,:\,V_i=\pi_j(1)}\nu^j_1\in\cM^+(\{V_i\}\times [0,\,T]), \qquad
		\omega:=\sum_{i\in\cW}\omega^i\in\cM^+(\cup_{i\in\cW}\{V_i\}\times [0,\,T]),
	\label{eq:omega}
\end{equation}
which represents the total mass flowing out of the network up to the time $T$.

Equation~\eqref{trdstep} takes then the form
\begin{equation}
	\dual{\mu}{\partial_t\varphi+v(x)\partial_x\varphi}=
		\dual{\mu_T-\mu_0}{\varphi}+\dual{\omega-\varsigma}{\varphi},
			\qquad \forall\,\varphi\in C^1(\Gamma\times [0,\,T]),
	\label{weakcond3}
\end{equation}
thereby expressing the counterpart of~\eqref{psexit} on the network.

\medskip

Using the formulation just obtained, we are in a position to establish the well-posedness of the transport problem over networks.
\begin{theorem} \label{thmgen}
Given $\mu_0\in\cM^+(\Gamma\times\{0\})$ and $\varsigma\in\cM^+(\cup_{i\in\cS}\{V_i\}\times [0,\,T])$, there exists a unique measure $\mu\in\cM^+(\Gamma\times [0,\,T])$ which satisfies the balance~\eqref{weakcond3} with $\mu_T\in\cM^+(\Gamma\times\{T\})$ defined in~\eqref{mu_j}-\eqref{eq:mu0T.network} and $\omega\in\cM^+(\cup_{i\in\cW}\{V_i\}\times [0,\,T])$ defined in~\eqref{nu_j}-\eqref{eq:omega}.

Moreover, for $\mu_{0,k}\in\cM^+(\Gamma\times\{0\})$, $\varsigma_k\in\cM^+(\cup_{i\in\cS}\{V_i\}\times [0,\,T])$, $k=1,\,2$, there exists a constant $C=C(T)>0$ such that
\begin{equation}
	\normBL{\mu_{T,2}-\mu_{T,1}}+\normBL{\omega_2-\omega_1}\leq C\left(\normBL{\mu_{0,2}-\mu_{0,1}}+\normBL{\varsigma_2-\varsigma_1}\right).
	\label{dependenceontheroad}
\end{equation}
\end{theorem}
\begin{proof}
We treat separately the cases in which the set of the source vertexes is or is not empty.

\begin{enumerate}[(i)]
\item Assume $\cS\ne\emptyset$. We introduce a partition of the set $\cE=\{E_j\}_{j\in J}$ based on the distance from the source set:
\begin{align*}
    \cE_0 &= \{E_j\,:\,V_i=\pi_j(0) \text{ is a source}\} \\
    \cE_m &= \{E_j\,:\,\exists\,E_k\in\cE_{m-1} \text{ s.t. } V_i=\pi_j(0)=\pi_k(1)\}, \qquad m=1,\,2,\,\dots
\end{align*}
We first apply Theorem~\ref{existence} to the problem defined on each arc in $\cE_{0}$, i.e for each $E_j\in\cE_0$ such that $V_i=\pi_j(0)$, $i\in \cS$, we consider
\begin{equation*}
	\begin{cases}
		\partial_t\mu^j+\partial_x(v_j(x)\mu^j)=0 & \text{in } E_j\times (0,\,T] \\[2mm]
		\mu^j_{t=0}=\mu_0^j\in\cM^+(E_j\times\{0\}) \\[2mm]
		\nu^j_0=p^i_j(t)\varsigma^i\in\cM^+(\{V_i\}\times [0,\,T]).
	\end{cases}
\end{equation*}
Since $\nu^j_0$ is prescribed, we obtain the existence of $\mu^j\in\cM^+(E_j\times [0,\,T])$, $\mu^j_T\in\cM^+(E_j\times\{T\})$ and $\nu^j_1\in\cM^+(\{\pi_j(1)\}\times [0,\,T])$ satisfying the balance~\eqref{psexit}. Next we proceed by induction on $m=1,\,2,\,\dots$ considering the problem on $E_j\in\cE_{m}$ with $V_i=\pi_j(0)$:
\begin{equation*}
	\begin{cases}
		\partial_t\mu^j+\partial_x(v_j(x)\mu^j)=0 & \text{in } E_j\times (0,\,T] \\[2mm]
		\mu^j_{t=0}=\mu_0^j\in\cM^+(E_j\times\{0\}) \\
		\nu^j_0=\sum	\limits_{k=1}^{d_I^i}p^i_{kj}(t)\nu_1^k\in\cM^+(\{V_i\}\times [0,\,T]).
	\end{cases}
\end{equation*}
Since the arcs $E_k$, $k=1,\,\dots,\,d_I^i$, belong to $\cE_{m-1}$, the solution to the transport equation on them is known by the inductive step (using the case $m=0$ as basis), hence the boundary measure $\nu^j_0$ is well defined because so are the outflow measures $\nu_1^k$. Therefore we can apply again Theorem~\ref{existence} to fulfil the balance~\eqref{psexit} on $E_j\in\cE_{m}$.

In this way, after a finite number of steps we build arc by arc the measures $\mu\in\cM^+(\Gamma\times [0,\,T])$, $\mu_T\in\cM^+(\Gamma\times\{T\})$ and $\omega\in\cM^+(\cup_{i\in\cW}\{V_i\}\times [0,\,T])$ which globally satisfy the balance~\eqref{weakcond3}.

\item Assume now $\cS=\emptyset$. Fix an arbitrary internal vertex $V_i$, $i\in\cI$, and choose
$$ t_0<\min_{j\in J\,:\,V_i=\pi_j(1)}\tau_j(0). $$
From~\eqref{nu_j} we see that, up to the time $t_0$, on all the arcs $E_j$ such that $V_i=\pi_j(1)$ the outflow measure $\nu^j_1$ is given by
$$ \nu^j_1=\int_{(\tau_j^{-1}(t_0),\,1]}\delta_{\tau_j(x)}\,d\mu^j_0(x), $$
because $\tau_j^{-1}(t_0)>0$ while $\sigma_j^{-1}(t_0)<0$ (cf. Figure~\ref{fig:characteristics}, left). Hence $\nu^j_1$ depends only on the initial datum $\mu_0^j$ and not on the inflow measure $\nu^j_{0}$.

Let us consider the initial/boundary-value problem~\eqref{problemonnet} for $t\in (0,\,t_0]$ with $V_i$ as source vertex and corresponding source measure
$$ \varsigma^i=\sum_{j\,:\,V_i=\pi_j(1)}\nu_1^j
	=\sum_{j\,:\,V_i=\pi_j(1)}\int_{(\tau_j^{-1}(t_0),\,1]}\delta_{\tau_j(x)}\,d\mu^j_0(x). $$
From the case $\cS\ne\emptyset$ we know that we can construct $\mu\in\cM^+(E_j\times [0,\,t_0])$, $\mu_{t_0}\in\cM(E_j\times\{t_0\})$ and $\omega\in\cM^+(\cup_{i\in\cW}\{V_i\}\times [0,\,t_0])$ which satisfy the balance~\eqref{psexit}. Moreover, the inflow measures $\nu^j_0$ of all the arcs $E_j$ such that $V_i=\pi_j(0)$ coincide with those of the original problem without sources, because they are actually determined only by the initial datum. Hence $\mu$ is also a solution of the original problem in $[0,\,t_0]$. By repeating this argument on the intervals $(t_0,\,2t_0]$, $(2t_0,\,3t_0]$, \dots, with initial data $\mu_{t_0}$, $\mu_{2t_0}$, \dots, after a finite number of steps we obtain the solution of the problem without source in any interval $[0,\,T]$, $T>0$.
\end{enumerate}

Finally, the estimate~\eqref{dependenceontheroad} is in both cases an immediate consequence of the corresponding estimate~\eqref{dependence} holding on each arc.
\end{proof}

\section{Examples of junctions}
\label{sec:examples}
In this section we write explicitly the solution to problem~\eqref{problemonnet} for two typical junctions which occur frequently for instance in traffic flow on road networks. It is worth pointing out that, since in our linear equation the velocity depends only on the space variable but not on the measure $\mu$ itself, the transport model that we are considering may provide an acceptable description of the flow of vehicles at most in the so-called \emph{free flow regime}. In fact, in such a case the number of vehicles is sufficiently small that their speed is almost independent of the presence of other vehicles on the road.

\subsection{The 1-2 junction -- Atomic inflow distribution} \label{sec:1-2_junct.atomic}
Let $\Gamma$ be the road network shown in Figure~\ref{fig:1-2_junction} formed by $3$ arcs, viz. roads, $E_1,\,E_2,\,E_3$ and $4$ vertexes $V_1,\,\dots,\,V_4$ such that $E_1$ connects the source vertex $V_1$ to the internal vertex $V_2$ while $E_2$ and $E_3$ connect the internal vertex $V_2$ to the well vertexes $V_3$ and $V_4$. This gives also the orientation of the arcs. In practice, beyond the junction $V_2$ the road $E_1$ splits in the two roads $E_2$, $E_3$.
\begin{figure}[!t]
\centering
\includegraphics[width=0.6\textwidth]{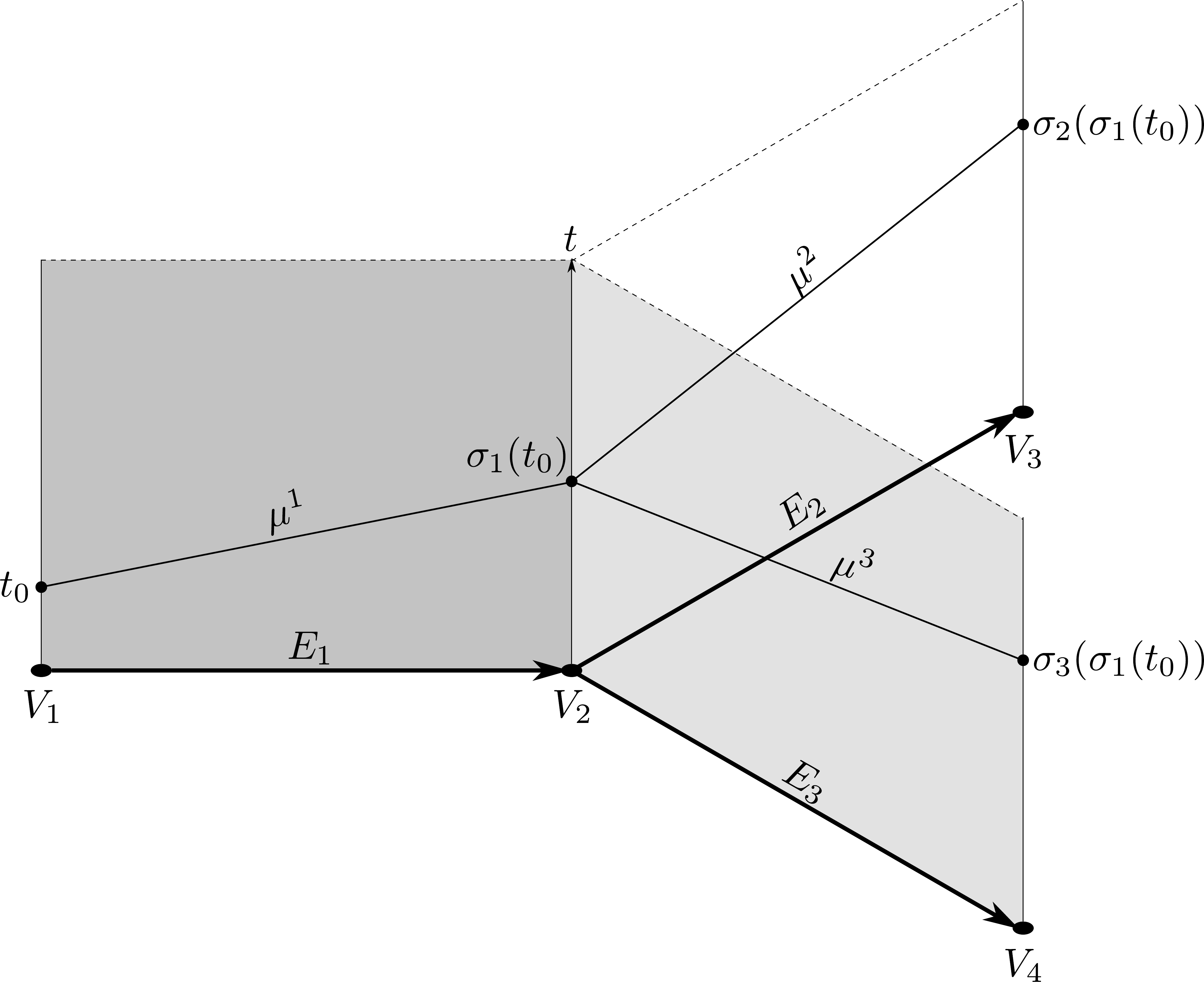}
\caption{The 1-2 junction with a sketch of the characteristics along which the solution to the example of Section~\ref{sec:1-2_junct.atomic} propagates in the space-time of the network.}
\label{fig:1-2_junction}
\end{figure}
We assume that the network is initially empty. At some time $t_0>0$ a \emph{microscopic} vehicle enters the network from the vertex $V_1$ and then travels across it. At the junction $V_2$ we prescribe a flux distribution rule stating that a time-dependent fraction $p=p(t):[0,\,T]\to [0,\,1]$ of the incoming mass flows to the road $E_2$ while the complementary fraction $1-p(t)$ flows to the road $E_3$. Taking $T=+\infty$, the problem can be formalised as:
\begin{equation*}
	\begin{cases}
		\partial_t\mu^j+\partial_x(v_j(x)\mu^j)=0 & x\in E_j,\,t\in\R^+,\,j=1,\,2,\,3 \\
		\mu_0=0 & x\in\Gamma \\
		\nu_0^1=\delta_{t_0} & t\in\R^+_0 \\
		\nu_0^2=p(t)\cdot\nu_1^1 & t\in\R^+_0 \\
		\nu_0^3=(1-p(t))\cdot\nu_1^1 & t\in\R^+_0,
	\end{cases}
\end{equation*}
where the velocity fields $v_j:E_j\to (0,\,\vmax^j]$, $0<\vmax^j<+\infty$, are given Lipschitz continuous functions of $x$.

The solution on each road has the form $\mu^j(dx\,dt)=\mu_t^j(dx)\otimes dt$, where $\mu^j_t$ is the trace of $\mu^j$ on the fibre $E_j\times\{t\}$. Using~\eqref{mu_j},~\eqref{nu_j} we determine explicitly the expression of $\mu^j_t$ for all $t>0$ and that of the outflow masses $\nu_1^j$ on the fibres $\{\pi_j(1)\}\times\R^+_0$ (notice that $\pi_1(1)=V_2$, $\pi_2(1)=V_3$, $\pi_3(1)=V_4$). We find (cf. Figure~\ref{fig:1-2_junction}):
\begin{align*}
	\mu^1_t &= \delta_{\Phi^1_t(0,\,t_0)}\chi_{[t_0,\,\sigma_1(t_0)]}(t), & \nu_1^1 &= \delta_{\sigma_1(t_0)} \\
	\mu^2_t &= p(\sigma_1(t_0))\delta_{\Phi^2_t(0,\,\sigma_1(t_0))}\chi_{[\sigma_1(t_0),\,\sigma_2(\sigma_1(t_0)]}(t),
		& \nu_1^2 &= \omega^3=p(\sigma_1(t_0))\delta_{\sigma_2(\sigma_1(t_0))} \\
	\mu^3_t &= [1-p(\sigma_1(t_0))]\delta_{\Phi^3_t(0,\,\sigma_1(t_0))}\chi_{[\sigma_1(t_0),\,\sigma_3(\sigma_1(t_0))]}(t),
		& \nu_1^3 &= \omega^4=[1-p(\sigma_1(t_0))]\delta_{\sigma_3(\sigma_1(t_0))}.
\end{align*}

Furthermore, using Bochner integrals in the product space $E_j\times\R^+_0$ we can possibly write the solution $\mu^j$ on each road as
\begin{align*}
	\mu^1 &= \int_{t_0}^{\sigma_1(t_0)}\delta_{(\Phi^1_t(0,\,t_0),\,t)}\,dt \\
	\mu^2 &= p(\sigma_1(t_0))\int_{\sigma_1(t_0)}^{\sigma_2(\sigma_1(t_0))}\delta_{(\Phi^2_t(0,\,\sigma_1(t_0)),\,t)}\,dt \\
	\mu^3 &= [1-p(\sigma_1(t_0))]\int_{\sigma_1(t_0)}^{\sigma_3(\sigma_1(t_0))}\delta_{(\Phi^3_t(0,\,\sigma_1(t_0)),\,t)}\,dt.
\end{align*}

\begin{remark}
By carefully inspecting the expressions of $\mu^j_t$, $j=1,\,2,\,3$, we see that the unit-mass Dirac delta prescribed at the source vertex $V_1$ splits in two Dirac deltas beyond the junction $V_2$, cf. also Figure~\ref{fig:1-2_junction}, each of which carries a fraction, $p(\sigma_1(t_0))$ and $1-p(\sigma_1(t_0))$, respectively, of the initial mass.

Unlike the Dirac delta entering the road $E_1$ from $V_1$, the two Dirac deltas propagating in the roads $E_2$, $E_3$ do \emph{not} represent physical microscopic vehicles. Rather, each of them is the same microscopic vehicle coming from the road $E_1$ and the coefficients $p(\sigma_1(t_0))$, $1-p(\sigma_1(t_0))$ have to be understood as the \emph{probabilities} that such a vehicle takes either outgoing road beyond the junction $V_2$.

This approach differs from the one proposed in~\cite{cristiani2016NHM}, which instead assigns a path to each microscopic vehicle through the network in the spirit of the multipath traffic model introduced in~\cite{bretti2014DCDSS,briani2014NHM}.
\end{remark}

\subsection{The 1-2 junction -- Continuous inflow distribution} \label{sec:1-2_junct.abscont}
We now consider the same network as in the previous Section~\ref{sec:1-2_junct.atomic} but we prescribe an inflow measure $\nu_0^1$ which is absolutely continuous with respect to the Lebesgue measure:
$$ \nu^1_0(dt):=\rho(t)\,dt, $$
where $\rho\in L^1(\R^+_0)$ with $\supp{\rho}\subseteq\R^+_0$ is the density of the vehicles entering the network from the vertex $V_1$.

Recalling that the network is initially empty and using~\eqref{mu_j}, we obtain that for each $t>0$ the trace $\mu^1_t$ of the solution $\mu^1$ in the road $E_1$ is
$$ \mu^1_t=\int_{\max\{0,\,\sigma_1^{-1}(t)\}}^t\delta_{\Phi^1_t(0,\,s)}\rho(s)\,ds
	=\int_0^{t-\max\{0,\,\sigma_1^{-1}(t)\}}\delta_{\Phi^1_r(0,\,0)}\rho(t-r)\,dr, $$
where in the last passage we have set $r:=t-s$ after observing from~\eqref{eq:char2} that $\Phi^1_t(0,\,s)=\Phi^1_{t-s}(0,\,0)$ for all $0\leq s\leq t$. Likewise, recalling~\eqref{nu_j} we find that the outflow mass $\nu^1_1$ at the vertex $V_2$ is
$$ \nu^1_1=\int_0^{+\infty}\delta_{\sigma_1(s)}\rho(s)\,ds
	=\int_{\sigma_1(0)}^{+\infty}\delta_r\rho(r-\sigma_1(0))\,dr, $$
where in the second passage we have set $r:=\sigma_1(s)=s+\sigma_1(0)$. In view of the Bochner representation~\eqref{superpos} and considering that $\supp{\rho(\cdot-\sigma_1(0))}\subseteq [\sigma_1(0),\,+\infty)$, we deduce in particular
$$ \nu^1_1(dt)=\rho(t-\sigma_1(0))\,dt. $$

According to our transmission conditions, this mass is distributed to the outgoing roads $E_2$, $E_3$ as
$$ \nu^2_0=p(t)\nu^1_1, \qquad \nu^3_0=(1-p(t))\nu^1_1, $$
which, owing to~\eqref{mu_j}, implies that the traces $\mu^2_t$, $\mu^3_t$ of the solutions $\mu^2$, $\mu^3$ in the outgoing roads are respectively given by
\begin{align*}
	\mu^2_t &= \int_{\max\{0,\,\sigma_2^{-1}(t)\}}^t\delta_{\Phi^2_t(0,\,s)}p(s)\rho(s-\sigma_1(0))\,ds \\
	&= \int_0^{t-\max\{0,\,\sigma_2^{-1}(t)\}}\delta_{\Phi^2_r(0,\,0)}p(t-r)\rho(t-r-\sigma_1(0))\,dr \\
	\intertext{and by}
	\mu^3_t &= \int_{\max\{0,\,\sigma_3^{-1}(t)\}}^t\delta_{\Phi^3_t(0,\,s)}(1-p(s))\rho(s-\sigma_1(0))\,ds \\
	&= \int_0^{t-\max\{0,\,\sigma_3^{-1}(t)\}}\delta_{\Phi^3_r(0,\,0)}(1-p(t-r))\rho(t-r-\sigma_1(0))\,dr.
\end{align*}
It is interesting to note that, since in general the density $\rho$ is split asymmetrically in the roads $E_2$ and $E_3$ (unless $p(t)=\frac{1}{2}$),
the corresponding measure solution, even if possibly continuous inside the arcs of the network, is discontinuous across the vertex $V_2$.

Finally, the outflow masses $\nu^2_1=\omega^3$ and $\nu^3_1=\omega^4$ are recovered from~\eqref{nu_j} as
\begin{align*}
	\nu^2_1=\omega^3 &= \int_0^{+\infty}\delta_{\sigma_2(s)}p(s)\rho(s-\sigma_1(0))\,ds \\
	&= \int_{\sigma_2(0)}^{+\infty}\delta_rp(r-\sigma_2(0))\rho(r-\sigma_1(0)-\sigma_2(0))\,dr
	\intertext{and}
	\nu^3_1=\omega^4 &= \int_0^{+\infty}\delta_{\sigma_3(s)}(1-p(s))\rho(s-\sigma_1(0))\,ds \\
	&= \int_{\sigma_3(0)}^{+\infty}\delta_r(1-p(r-\sigma_3(0)))\rho(r-\sigma_1(0)-\sigma_3(0))\,dr.
\end{align*}
Observing that $\supp{\rho(\cdot-\sigma_1(0)-\sigma_j(0))}\subseteq [\sigma_1(0)+\sigma_j(0),\,+\infty)$ for $j=2,\,3$, from the Bochner representation~\eqref{superpos} of a measure we further deduce
\begin{align*}
	\nu^2_1(dt)=\omega^3(dt) &= p(t-\sigma_2(0))\rho(t-\sigma_1(0)-\sigma_2(0))\,dt \\
	\nu^3_1(dt)=\omega^4(dt) &= (1-p(t-\sigma_3(0))\rho(t-\sigma_1(0)-\sigma_3(0))\,dt.
\end{align*}

\begin{remark}
The transport problem being linear, the case of an inflow measure $\nu^1_0$ carrying both an atomic and a Lebesgue-absolutely continuous part can be addressed by simply superimposing the solutions obtained in Sections~\ref{sec:1-2_junct.atomic} and~\ref{sec:1-2_junct.abscont}.
\end{remark}

\subsection{The 2-1 junction}
\label{sec:2-1_junct}
\begin{figure}[!t]
\centering
\includegraphics[width=0.6\textwidth]{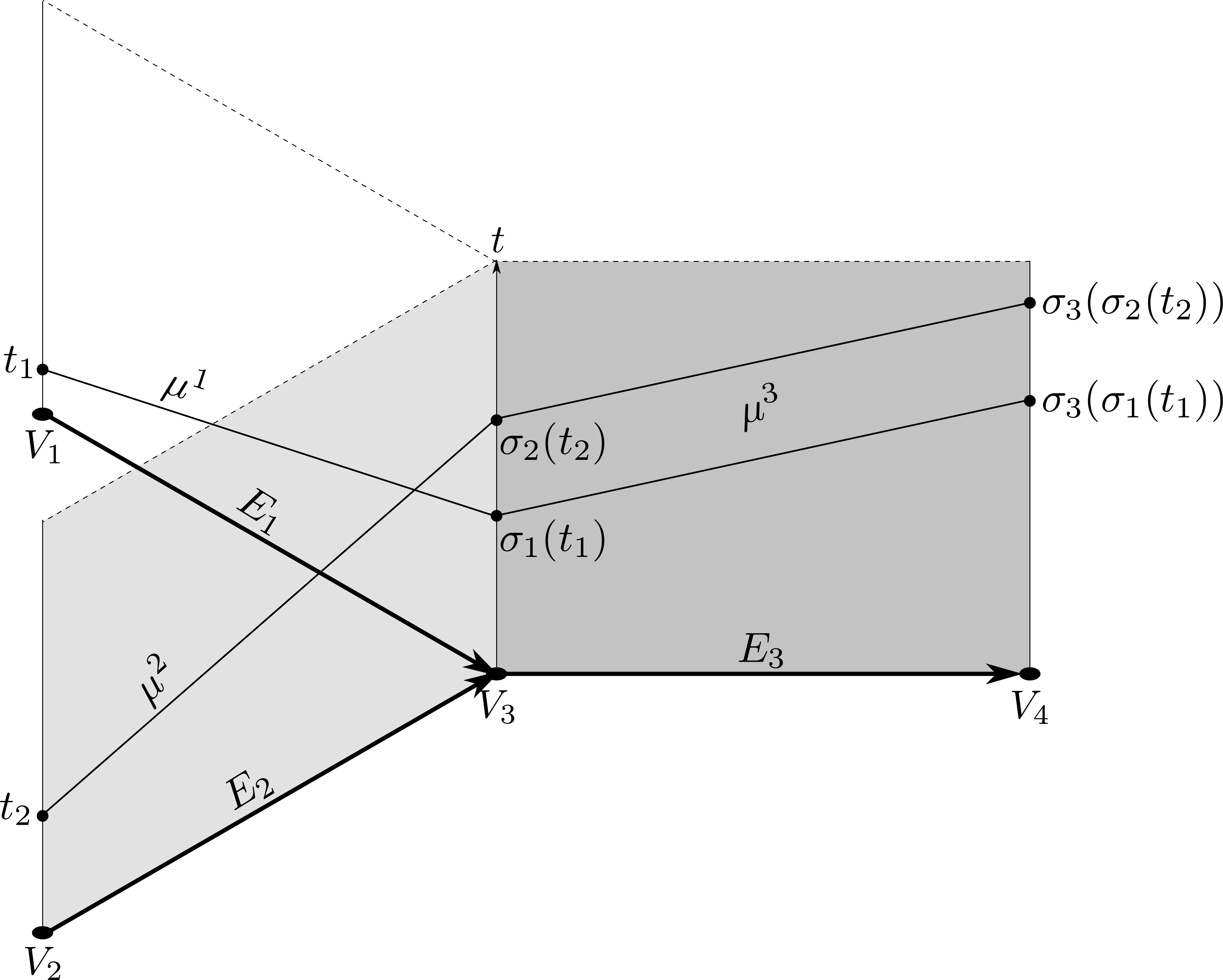}
\caption{The 2-1 junction with a sketch of the characteristics along which the solution to the example of Section~\ref{sec:2-1_junct} propagates in the space-time of the network.}
\label{fig:2-1_junction}
\end{figure}
We consider now the road network $\Gamma$ illustrated in Figure~\ref{fig:2-1_junction} with again $3$ arcs, viz. roads, $E_1,\,E_2,\,E_3$ and $4$ vertexes $V_1,\,\dots,\,V_4$. However, in this case both vertexes $V_1$, $V_2$ are sources and are connected by roads $E_1$, $E_2$ to the internal vertex $V_3$. The latter is finally connected to the well vertex $V_4$ by road $E_3$. In practice, beyond the junction $V_3$ the incoming roads $E_1$, $E_2$ merge into the outgoing road $E_3$.

Like in Sections~\ref{sec:1-2_junct.atomic},~\ref{sec:1-2_junct.abscont}, we assume that the network is initially empty. At two successive time instants $0\leq t_1\leq t_2$ two \emph{microscopic} vehicles enter the network from the sources $V_1$, $V_2$, respectively. Their propagation across the network for $t>0$ is then described by the problem:
\begin{equation*}
	\begin{cases}
		\partial_t\mu^j+\partial_x(v_j(x)\mu^j)=0 & x\in E_j,\,t\in\R^+,\,j=1,\,2,\,3 \\
		\mu_0=0 & x\in\Gamma \\
		\nu_0^1=\delta_{t_1} & t\in\R^+_0 \\
		\nu_0^2=\delta_{t_2} & t\in\R^+_0 \\
		\nu_0^3=\nu_1^1+\nu_1^2 & t\in\R^+_0,
	\end{cases}
\end{equation*}
where the velocity fields $v_j:E_j\to (0,\,\vmax^j]$, $0<\vmax^j<+\infty$, are as usual given Lipschitz continuous functions of $x$. Notice that, for mass conservation purposes, the flux distribution coefficients at the junction $V_3$ are necessarily $p^3_{13}(t)=p^3_{23}(t)=1$ for all $t>0$.

Relying again on~\eqref{mu_j},~\eqref{nu_j} we write explicitly the solution $\mu^j\in\cM^+(E_j\times\R^+_0)$ on each road as well as the outflow measures $\nu_1^j\in\cM^+(\{\pi_j(1)\}\times\R^+_0)$, with $\pi_1(1)=\pi_2(1)=V_3$ and $\pi_3(1)=V_4$. We find (cf. Figure~\ref{fig:2-1_junction}):
\begin{align*}
	\mu^1_t &= \delta_{\Phi^1_t(0,\,t_1)}\chi_{[t_1,\,\sigma_1(t_1)]}(t), & \nu_1^1 &= \delta_{\sigma_1(t_1)} \\
	\mu^2_t &= \delta_{\Phi^2_t(0,\,t_2)}\chi_{[t_2,\,\sigma_2(t_2)]}(t), & \nu_1^2 &= \delta_{\sigma_2(t_2)} \\
	\mu^3_t &= \delta_{\Phi^3_t(0,\,\sigma_1(t_1))}\chi_{[\sigma_1(t_1),\,\sigma_3(\sigma_1(t_1))]}(t) &
		\nu_1^3 &= \omega^4=\delta_{\sigma_3(\sigma_1(t_1))}+\delta_{\sigma_3(\sigma_2(t_2))}, \\
	&\phantom{=} +\delta_{\Phi^3_t(0,\,\sigma_2(t_2))}\chi_{[\sigma_2(t_2),\,\sigma_3(\sigma_2(t_2))]}(t),
\end{align*}
whence, using Bochner integrals in the product spaces $E_j\times\R^+_0$, $j=1,\,2,\,3$,
\begin{align*}
	\mu^1 &= \int_{t_1}^{\sigma_1(t_1)}\delta_{(\Phi^1_t(0,\,t_1),\,t)}\,dt \\
	\mu^2 &= \int_{t_2}^{\sigma_2(t_2)}\delta_{(\Phi^2_t(0,\,t_2),\,t)}\,dt \\
	\mu^3 &= \int_{\sigma_1(t_1)}^{\sigma_3(\sigma_1(t_1))}\delta_{(\Phi^3_t(0,\,\sigma_1(t_1)),\,t)}\,dt
		+\int_{\sigma_2(t_2)}^{\sigma_3(\sigma_2(t_2))}\delta_{(\Phi^3_t(0,\,\sigma_2(t_2)),\,t)}\,dt.
\end{align*}

\appendix

\section{Proofs of the theorems of Section~\ref{sec:interval}}
\label{appendix}
\begin{proof}[Proof of Theorem~\ref{existence}]
We observe that $\mu$ can be obtained, by linearity, as the sum of the solutions of two transport problems with $\nu_0=0$ and $\mu_{0}=0$, respectively.

We begin by considering the case $\nu_0=0$ and assume, without loss of generality, that $T\leq\tau(0)$. Then $\tau^{-1}(T)\geq 0$ whence, recalling~\eqref{candidate},~\eqref{outnu}, we obtain
\begin{equation}
	\mu_T=\int_{[0,\,\tau^{-1}(T)]}\delta_{\Phi_{T}(x,\,0)}\,d\mu_{0}(x), \qquad
		\nu_1=\int_{(\tau^{-1}(T),\,1]}\delta_{\tau(x)}\,d\mu_{0}(x)
	\label{eq:proof.restrictions.1}
\end{equation}
and we have to show that
\begin{equation}
	\dual{\mu\llcorner Q}{\partial_t\varphi+v(x)\partial_x\varphi}=\dual{\mu_T-\mu_0}{\varphi}+\dual{\nu_1}{\varphi},
		\qquad \forall\,\varphi\in C^1(\bar{Q}),
	\label{P1}
\end{equation}
where $\mu$ is the measure~\eqref{solcand1}. Following the characteristics, its restriction to $Q$ writes as
$$ \mu\llcorner Q(dx\,dt)=\underbrace{\int_{[0,\,\tau^{-1}(t)]}\delta_{\Phi_t(\xi,\,0)}(dx)\,d\mu_0(\xi)}_{:=\mu_t\llcorner Q(dx)}\otimes\,dt, $$
thus for $\varphi\in C^1(\bar{Q})$ we discover:
\begin{align*}
	\dual{\mu\llcorner Q}{\partial_t\varphi+v(x)\partial_x\varphi}
		&= \int_0^T\int_{[0,\,1]}\left(\partial_t\varphi+v(x)\partial_x\varphi\right)\,d\mu_t\llcorner Q(x)\,dt \\
	&= \int_0^T\int_{[0,\,\tau^{-1}(t)]}\Bigl(\partial_t\varphi(\Phi_t(x,\,0),\,t)+v(\Phi_t(x,\,0))\partial_x\varphi(\Phi_t(x,\,0),\,t)\Bigr)\,d\mu_0(x)\,dt \\
	&= \int_0^T\int_{[0,\,\tau^{-1}(t)]}\frac{d}{dt}\varphi(\Phi_t(x,\,0),\,t)\,d\mu_0(x)\,dt,
\intertext{where in the last passage we have used~\eqref{eq:char1}. Switching the order of integration, we continue the calculation as:}
	&= \int_{[0,\,1]}\int_0^{\min\{\tau(x),\,T\}}\frac{d}{dt}\varphi(\Phi_t(x,\,0),\,t)\,dt\,d\mu_0(x) \\
	&= \int_{[0,\,\tau^{-1}(T)]}\int_0^T\frac{d}{dt}\varphi(\Phi_t(x,\,0),\,t)\,dt\,d\mu_0(x) \\
	&\phantom{=} +\int_{(\tau^{-1}(T),\,1]}\int_0^{\tau(x)}\frac{d}{dt}\varphi(\Phi_t(x,\,0),\,t)\,dt\,d\mu_0(x) \\
	&= \int_{[0,\,\tau^{-1}(T)]}\Bigl(\varphi(\Phi_T(x,\,0),\,T)-\varphi(\Phi_0(x,\,0),\,0)\Bigr)\,d\mu_0(x) \\
	&\phantom{=} +\int_{(\tau^{-1}(T),\,1]}\Bigl(\varphi(\Phi_{\tau(x)}(x,\,0),\,\tau(x))-\varphi(\Phi_0(x,\,0),\,0)\Bigr)\,d\mu_0(x) \\
	&= \underbrace{\int_{[0,\,\tau^{-1}(T)]}\varphi(\Phi_T(x,\,0),\,T)\,d\mu_0(x)}_{\textrm{(i)}}
		+\underbrace{\int_{(\tau^{-1}(T),\,1]}\varphi(1,\,\tau(x))\,d\mu_0(x)}_{\textrm{(ii)}} \\
	&\phantom{=}	-\underbrace{\int_{[0,\,1]}\varphi(x,\,0)\,d\mu_0(x)}_{\textrm{(iii)}}.
\end{align*}
At this point, from~\eqref{eq:proof.restrictions.1} we recognise that the term (i) is indeed $\int_{[0,\,1]}\varphi(x,\,T)\,d\mu_T(x)=\dual{\mu_T}{\varphi}$ and that the term (ii) is $\int_{[0,\,T]}\varphi(1,\,t)\,d\nu_1(t)=\dual{\nu_1}{\varphi}$, while the term (iii) is clearly $\dual{\mu_0}{\varphi}$. Consequently~\eqref{P1} follows.

\medskip

We consider now the case $\mu_0=0$ and assume, without loss of generality, that $T\geq\sigma(0)$. Then $\sigma^{-1}(T)\geq 0$ whence, recalling again~\eqref{candidate},~\eqref{outnu}, we find
\begin{equation}
	\mu_T=\int_{[\sigma^{-1}(T),\,T]}\delta_{\Phi_T(0,\,s)}\,d\nu_0(s), \qquad
		\nu_1=\int_{[0,\,\sigma^{-1}(T))}\delta_{\sigma(s)}\,d\nu_0(s)
	\label{eq:proof.restrictions.2}
\end{equation}
and we have to show that
\begin{equation}
	\dual{\mu\llcorner Q}{\partial_t\varphi+v(x)\partial_x\varphi}=\dual{\mu_T}{\varphi}+\dual{\nu_1-\nu_0}{\varphi},
		\qquad \forall\,\varphi\in C^1(\bar{Q}),
	\label{P2}
\end{equation}
where $\mu$ is again the measure~\eqref{solcand1}. Following the characteristics we see that $\mu\llcorner Q$ is now expressed as
$$ \mu\llcorner Q(dx\,dt)=\underbrace{\int_{[\max\{0,\,\sigma^{-1}(t)\},\,t]}\delta_{\Phi_t(0,\,s)}(dx)\,d\nu_0(s)}_{:=\mu_t\llcorner Q(dx)}\otimes\,dt, $$
hence for $\varphi\in C^1(\bar{Q})$ we obtain:
\begin{align*}
	\dual{\mu\llcorner Q}{\partial_t\varphi+v(x)\partial_x\varphi} &=
		\int_0^T\int_{[0,\,1]}\left(\partial_t\varphi+v(x)\partial_x\varphi\right)\,d\mu_t\llcorner Q(x)\,dt \\
	&= \int_0^T\int\limits_{[\max\{0,\,\sigma^{-1}(t)\},\,t]}\hspace{-5mm}\Bigl(\partial_t\varphi(\Phi_t(0,\,s),\,t)+
		v(\Phi_t(0,\,s))\partial_x\varphi(\Phi_t(0,\,s)\,t)\Bigr)\,d\nu_0(s)\,dt \\
	&= \int_0^T\int_{[\max\{0,\,\sigma^{-1}(t)\},\,t]}\frac{d}{dt}\varphi(\Phi_t(0,\,s),\,t)\,d\nu_0(s)\,dt,
\intertext{where in the last passage we have used~\eqref{eq:char2}. We now switch the order of integration to discover:}
	&= \int_{[0,\,T]}\int_s^{\min\{\sigma(s),\,T\}}\frac{d}{dt}\varphi(\Phi_t(0,\,s),\,t)\,dt\,d\nu_0(s) \\
	&= \int_{[0,\,\sigma^{-1}(T)]}\int_s^{\sigma(s)}\frac{d}{dt}\varphi(\Phi_t(0,\,s),\,t)\,dt\,d\nu_0(s) \\
	&\phantom{=} +\int_{(\sigma^{-1}(T),\,T]}\int_s^T\frac{d}{dt}\varphi(\Phi_t(0,\,s),\,t)\,dt\,d\nu_0(s) \\
	&= \int_{[0,\,\sigma^{-1}(T)]}\Bigl(\varphi(\Phi_{\sigma(s)}(0,\,s),\,\sigma(s))-\varphi(\Phi_s(0,\,s),\,s)\Bigr)\,d\nu_0(s) \\
	&\phantom{=} +\int_{(\sigma^{-1}(T),\,T]}\Bigl(\varphi(\Phi_T(0,\,s)\,T)-\varphi(\Phi_s(0,\,s),\,s)\Bigr)\,d\nu_0(s) \\
	&= \underbrace{\int_{[0,\,\sigma^{-1}(T)]}\varphi(1,\,\sigma(s))\,d\nu_0(s)}_{\textrm{(i)}}
		+\underbrace{\int_{(\sigma^{-1}(T),\,T]}\varphi(\Phi_T(0,\,s),\,T)\,d\nu_0(s)}_{\textrm{(ii)}} \\
	&\phantom{=} -\underbrace{\int_{[0,\,T]}\varphi(0,\,s)\,d\nu_0(s)}_{\textrm{(iii)}}.
\end{align*}
Thanks to~\eqref{eq:proof.restrictions.2} we recognise that the term (i) is $\int_{[0,\,T]}\varphi(1,\,t)\,d\nu_1(t)=\dual{\nu_1}{\varphi}$ and that the term (ii) is $\int_{[0,\,1]}\varphi(x,\,T)\,d\mu_T(x)=\dual{\mu_T}{\varphi}$, while the term (iii) is clearly $\dual{\nu_0}{\varphi}$. Hence~\eqref{P2} follows.

\bigskip

To conclude the proof, we show the continuous dependence estimate~\eqref{dependence}. We consider two problems of the type~\eqref{problemsingleQ} with respective initial data $\mu^1_0,\,\mu^2_0$ and source data $\nu^1_0,\,\nu^2_0$.

We begin by estimating the term $\normBL{\mu^2_T-\mu^1_T}$. Let $\varphi\in BL(Q)$ with $\Vert\varphi\Vert_{BL}\leq 1$. Recalling~\eqref{candidate} we have:
\begin{align*}
	\dual{\mu^2_T-\mu^1_T}{\varphi} &= \int_{[0,\,1]}\varphi(x,\,T)\,d(\mu^2_T-\mu^1_T)(x) \\
	&= \int_{[0,\,\max\{0,\,\tau^{-1}(T)\}]}\varphi(\Phi_T(x,\,0),\,T)\,d(\mu^2_0-\mu^1_0)(x) \\
	&\phantom{=} +\int_{[\max\{0,\,\sigma^{-1}(T)\},\,T]}\varphi(\Phi_T(0,\,s),\,T)\,d(\nu^2_0-\nu^1_0)(s) \\
	&\leq \abs{\mu^2_0-\mu^1_0}([0,\,\max\{0,\,\tau^{-1}(T)\}])+
		\abs{\nu^2_0-\nu^1_0}([\max\{0,\,\sigma^{-1}(T)\},\,T])
\intertext{where here $\abs{\cdot}$ stands for the total variation of a measure. Thus}
	&\leq C\left(\normBL{\mu^2_0-\mu^1_0}+\normBL{\nu^2_0-\nu^1_0}\right)
\end{align*}
and consequently, taking the supremum over $\varphi$ at both sides,
$$ \normBL{\mu^2_T-\mu^1_T}\leq C\left(\normBL{\mu^2_0-\mu^1_0}+\normBL{\nu^2_0-\nu^1_0}\right). $$

Proceeding in a similar way for $\normBL{\nu^2_1-\nu^1_1}$, from~\eqref{outnu} we have:
\begin{align*}
	\dual{\nu^2_1-\nu^1_1}{\varphi} &= \int_{[0,\,T]}\varphi(1,\,t)\,d(\nu^2_1-\nu^1_1)(t) \\
	&= \int_{(\max\{0,\,\tau^{-1}(T)\},\,1]}\varphi(1,\,\tau(x))\,d(\mu^2_0-\mu^1_0)(x) \\
	&\phantom{=} +\int_{[0,\,\max\{0,\,\sigma^{-1}(T)\})}\varphi(1,\,\sigma(s))\,d(\nu^2_0-\nu^1_0)(s) \\
	&\leq \abs{\mu^2_0-\mu^1_0}((\max\{0,\,\tau^{-1}(T)\},\,1])+
		\abs{\nu^2_0-\nu^1_0}([0,\,\max\{0,\,\sigma^{-1}(T)\})) \\
	&\leq C\left(\normBL{\mu^2_0-\mu^1_0}+\normBL{\nu^2_0-\nu^1_0}\right),
\end{align*}
hence, taking the supremum over $\varphi$ at both sides,
$$ \normBL{\nu^2_1-\nu^1_1}\leq C\left(\normBL{\mu^2_0-\mu^1_0}+\normBL{\nu^2_0-\nu^1_0}\right). $$

Summing the two estimates just obtained yields finally~\eqref{dependence}.

Moreover, for $\mu^1_0=\mu^2_0$, $\nu^1_0=\nu^2_0$ the estimate~\eqref{dependence} implies $\mu^1_T=\mu^2_T$, $\nu^1_1=\nu^2_1$, hence the uniqueness of~\eqref{candidate} and~\eqref{outnu}.
\end{proof}

\begin{proof}[Proof of Theorem~\ref{theo:time_dep}]
We begin with the estimate of $\normBL{\mu_t-\mu_{t'}}$. Let $\varphi\in BL(Q)$ be such that $\Vert\varphi\Vert_{BL}\leq 1$. By~\eqref{candidate}, since
$$ (-\infty,\,\tau^{-1}(t'))=(-\infty,\,\tau^{-1}(t))\cup [\tau^{-1}(t),\,\tau^{-1}(t')], $$
we can write:
\begin{align*}
	& \int_{(-\infty,\,\tau^{-1}(t))\cap [0,\,1)}\varphi(\Phi_{t}(x,\,0),\,t)\,d\mu_0(x)
		-\int_{(-\infty,\,\tau^{-1}(t'))\cap [0,\,1)}\varphi(\Phi_{t'}(x,\,0),\,t')\,d\mu_0(x) \\
	&= \int_{(-\infty,\,\tau^{-1}(t))\cap [0,\,1]}\Bigl(\varphi(\Phi_{t}(x,\,0),\,t)
		-\varphi(\Phi_{t'}(x,\,0),\,t')\Bigr)\,d\mu_0(x) \\
	&\phantom{=}	-\int_{[\tau^{-1}(t),\,\tau^{-1}(t'))\cap [0,\,1]}\varphi(\Phi_{t'}(x,\,0),\,t')\,d\mu_0(x) \\
	&\leq \mu_{0}((-\infty,\,\tau^{-1}(t))\cap [0,\,1))\norm{v}_\infty\abs{t-t'}
		-\int_{[\tau^{-1}(t),\,\tau^{-1}(t'))\cap [0,\,1]}\varphi(\Phi_{t'}(x,\,0),\,t')\,d\mu_0(x).
\end{align*}

Likewise, assuming for simplicity that $\sigma^{-1}(t)\leq t'$,
\begin{align*}
	& \int_{(\sigma^{-1}(t),\,t]\cap (0,\,T]}\varphi(\Phi_{t}(0,\,s),\,t)\,d\nu_0(s)
		-\int_{(\sigma^{-1}(t'),\,t']\cap (0,\,T]}\varphi(\Phi_{t}(0,\,s),\,t)\,d\nu_0(s) \\
	&= -\int_{(\sigma^{-1}(t'),\,\sigma^{-1}(t)]}\varphi(\Phi_{t'}(0,\,s),\,t')\,d\nu_0(s) \\
	&\phantom{=} +\int_{(\sigma^{-1}(t),\,t']}\Bigl(\varphi(\Phi_{t}(0,\,s),\,t)-\varphi(\Phi_{t'}(0,\,s),\,t')\Bigr)\,d\nu_0(s)
		+\int_{(t',\,t]}\varphi(\Phi_{t}(0,\,s),\,t)\,d\nu_0(s) \\
	&\leq \nu_0((t',\,t])+\nu_0((t-\tau(0),\,t'])\norm{v}_\infty\abs{t-t'}
		-\int_{(\sigma^{-1}(t'),\,\sigma^{-1}(t)]}\varphi(\Phi_{t'}(0,\,s),\,t')\,d\nu_0(s).
\end{align*}
Hence
\begin{align*}
	\abs{\dual{\mu_t-\mu_{t'}}{\varphi}} &\leq
		\abs{\int_{(\tau^{-1}(t),\,\tau^{-1}(t')]\cap [0,\,1]}\Bigl(\varphi(\Phi_{t'}(x,\,0),\,t')
			-\varphi(1,\,\tau(x))\Bigr)\,d\mu_0(x)} \\
	&\phantom{\leq} +\abs{\int_{(\sigma^{-1}(t'),\,\sigma^{-1}(t)]}\Bigl(\varphi(1,\,\sigma(s))-\varphi(\Phi_{t'}(0,\,s),\,t')\Bigr)\,d\nu_0(s)} \\
	&\leq \mu_0((\tau^{-1}(t),\,\tau^{-1}(t')]\cap [0,\,1])\norm{v}_\infty\abs{t-t'} \\
	&\phantom{\leq} +\nu_0((\sigma^{-1}(t'),\,\sigma^{-1}(t)])\norm{v}_\infty\abs{t-t'} \\
	&\leq \norm{v}_\infty\Bigl(\mu_{0}([0,\,1])+\nu_0([0,\,t])\Bigr)\abs{t-t'}+\nu_0((t',\,t]) \\
	&\leq C\abs{t-t'}+\nu_0([t',\,t])
\end{align*}
and finally, taking the supremum over $\varphi$ at both sides,
$$ \normBL{\mu_t-\mu_{t'}}\leq C\abs{t-t'}+\nu_0([t',\,t]). $$

We now consider the estimate on the outflow measures. Taking again $\varphi\in BL(Q)$ with $\norm{\varphi}_{BL}\leq 1$, we compute:
\begin{align*}
	\dual{\nu_1\llcorner [0,\,t]-\nu_1\llcorner [0,\,t']}{\varphi} &= \int_{[0,\,1)\cap[\tau^{-1}(t),\,1)}\varphi(1,\,\tau(x))\,d\mu_0(x)
		+\int_{(0,\,t]\cap (0,\,\sigma^{-1}(t)]}\varphi(1,\,\sigma(s))\,d\nu_0(s) \\
	&\phantom{=} -\int_{[0,\,1)\cap [\tau^{-1}(t'),\,1)}\varphi(1,\,\tau(x))\,d\mu_0(x)
		-\int_{(0,\,t']\cap (0,\,\sigma^{-1}(t')]}\hspace{-3mm}\varphi(1,\,\sigma(s))\,d\nu_0(s).
\end{align*}
We point out that if $\sigma^{-1}(t)<0$ then the interval $(0,\,\sigma^{-1}(t)]$ is actually understood as $[\sigma^{-1}(t),\,0)$ and, in this case, $(0,\,t]\cap (0,\,\sigma^{-1}(t)]=\emptyset$. Moreover, since $t>t'$ we have $\tau^{-1}(t')>\tau^{-1}(t)$, which implies $[\tau^{-1}(t'),\,1)=[\tau^{-1}(t'),\,\tau^{-1}(t))\cup [\tau^{-1}(t),\,1)$. Then
\begin{multline*}
	\int_{[0,\,1)\cap [\tau^{-1}(t),\,1)}\varphi(1,\,\tau(x))\,d\mu_{0}(x)
		-\int_{[0,\,1)\cap [\tau^{-1}(t'),\,1)}\varphi(1,\,\tau(x))\,d\mu_{0}(x) \\
	=\int_{[0,\,1)\cap [\tau^{-1}(t),\,\tau^{-1}(t'))}\varphi(1,\,\tau(x))\,d\mu_0(x).
\end{multline*}
Moreover,
\begin{multline*}
	\int_{(0,\,\sigma^{-1}(t)]\cap (0,\,t]}\varphi(1,\,\sigma(s))\,d\nu_0(s) \\
	=\int_{(0,\,\sigma^{-1}(t')]\cap (0,\,t]}\varphi(1,\,\sigma(s))\,d\nu_0(s)
		+\int_{(\sigma^{-1}(t'),\,\sigma^{-1}(t)]\cap (0,\,t]}\varphi(1,\,\sigma(s))\,d\nu_0(s),
\end{multline*}
which gives
\begin{multline*}
	\int_{(0,\,\sigma^{-1}(t)]\cap (0,\,t]}\varphi(1,\,\sigma(s))\,d\nu_0(s)
		-\int_{(0,\,\sigma^{-1}(t')]\cap (0,\,t']}\varphi(1,\,\sigma(s))\,d\nu_0(s) \\
	=\int_{(\sigma^{-1}(t'),\,\sigma^{-1}(t)]\cap (0,\,t]}\varphi(1,\,\sigma(s))\,d\nu_0(s).
\end{multline*}
Therefore
\begin{align*}
	\dual{\nu_1\llcorner [0,\,t]-\nu_1\llcorner [0,\,t']}{\varphi} &= \int_{[0,\,1)\cap [\tau^{-1}(t),\,\tau^{-1}(t'))}\varphi(1,\,\tau(x))\,d\mu_0(x) \\
	&\phantom{=} +\int_{(\sigma^{-1}(t'),\,\sigma^{-1}(t)]\cap (0,\,t]}\varphi(1,\,\sigma(s))\,d\nu_0(s) \\
	&\leq \nu_0((t',\,t])+\nu_0((\sigma^{-1}(t),\,t'])\norm{v}_\infty\abs{t-t'} \\
	&\phantom{\leq}	+\mu_0((-\infty,\,\tau^{-1}(t))\cap [0,\,1))\norm{v}_\infty\abs{t-t'},
\end{align*}
whence, taking the supremum over $\varphi$ at both sides,
$$ \normBL{\nu_1\llcorner [0,\,t]-\nu_1\llcorner [0,\,t']}\leq C\abs{t-t'}+\nu_0([t',\,t]). $$

Summing the estimates obtained so far for the terms $\normBL{\mu_t-\mu_{t'}}$, $\normBL{\nu_1\llcorner [0,\,t]-\nu_1\llcorner [0,\,t']}$ we finally get~\eqref{notcont}.
\end{proof}

\section*{Acknowledgements}
A.T. is member of GNFM (Gruppo Nazionale per la Fisica Matematica) of INdAM (Istituto Nazionale di Alta Matematica), Italy.

A.T. acknowledges that this work has been written within the activities of a research project funded by ``Compagnia di San Paolo'' (Turin, Italy).

\bibliographystyle{amsplain}
\bibliography{references}
\end{document}